\documentclass[10pt]{amsart}

\usepackage{amsmath,amssymb,amscd,amsthm,amsxtra}

\usepackage{graphicx, mathrsfs}
\usepackage{amssymb,amsmath,amsthm}
\usepackage{mathrsfs,dsfont}

\makeatletter
\DeclareRobustCommand\widecheck[1]{{\mathpalette\@widecheck{#1}}}
\def\@widecheck#1#2{%
   \setbox\z@\hbox{\m@th$#1#2$}%
   \setbox\tw@\hbox{\m@th$#1%
      \widehat{%
         \vrule\@width\z@\@height\ht\z@
         \vrule\@height\z@\@width\wd\z@}$}%
   \dp\tw@-\ht\z@
   \@tempdima\ht\z@ \advance\@tempdima2\ht\tw@ \divide\@tempdima\thr@@
   \setbox\tw@\hbox{%
      \raise\@tempdima\hbox{\scalebox{1}[-1]{\lower\@tempdima\box\tw@}}}%
   {\ooalign{\box\tw@ \cr \box\z@}}}
\makeatother

\newtheorem{theorem}{Theorem} [section]

\newtheorem{lemma}[theorem]{Lemma}
\newtheorem{proposition}[theorem]{Proposition}
\newtheorem{remark}[theorem]{Remark}

\newtheorem{corollary}[theorem]{Corollary}

\begin{document}

\title[Bounds on Sobolev norms for NLS on $\mathbb{R}$]{Bounds on the growth of high Sobolev norms of solutions to Nonlinear Schr\"{o}dinger Equations on $\mathbb{R}$}
\author{Vedran Sohinger}
\address{Massachusetts Institute of Technology, Department of Mathematics}
\email{\tt vedran@math.mit.edu}

\maketitle
\begin{abstract}
In this paper, we consider the cubic nonlinear Schr\"{o}dinger equation, and the Hartree equation, with sufficiently regular convolution potential, both on the real line. We are interested in bounding the growth of high Sobolev norms of solutions to these equations. Since the cubic NLS is completely integrable, it makes sense to bound only the fractional Sobolev norms of solutions, whose initial data is of restricted smoothness. For the Hartree equation, we consider all Sobolev norms. For both equations, we derive our results by using an appropriate frequency decomposition. In the case of the cubic NLS, this method allows us to recover uniform bounds on the integral Sobolev norms, up to a factor of $t^{0+}$. For the Hartree equation, we use the same method as in our previous paper \cite{SoSt1}, and the improved Strichartz estimate to obtain a better bound than the one that was obtained in the periodic setting in the mentioned work.
\end{abstract}

\section{Introduction.}

\subsection{Statement of the problem and of the main results:}

Given $s \geq 1$ real, we study the defocusing cubic nonlinear Schr\"{o}dinger initial value problem on $\mathbb{R}$:

\begin{equation}
\label{eq:cubicnls}
\begin{cases}
iu_t + \Delta u = |u|^2u,\, x \in \mathbb{R},\, t \in \mathbb{R}\\
u|_{t=0}= \Phi \in H^s(\mathbb{R}).
\end{cases}
\end{equation}

\vspace{3mm}

The equation (\ref{eq:cubicnls}) arises in the Gross-Pitaevski scaling limit of large systems of bosons, and in geometric optics \cite{Sch,SuSu,ZM}. The one-dimensional cubic NLS on $\mathbb{R}$ has a specific physical meaning, and it is used to describe a Bose gas in elongated traps and the so-called cigar-shaped Bose-Einstein condensates \cite{LSY}. For a rigorous mathematical derivation of the equation from many-body quantum dynamics, the reader should consult \cite{AGT}. A rigorous derivation of the two-dimensional version can be found in \cite{KSchSta}.

\vspace{2mm}

Furthermore, we study the Hartree initial value problem on $\mathbb{R}$

\begin{equation}
\label{eq:Hartree}
\begin{cases}
i u_t + \Delta u=(V*|u|^2)u, x \in \mathbb{R}, t \in \mathbb{R}\\
u|_{t=0}=\Phi \in H^s(\mathbb{R}).
\end{cases}
\end{equation}

\vspace{2mm}

The assumptions that we have on $V$ are the following:

\begin{enumerate}
\item[(i)] $V \in L^1(\mathbb{R})$
\item[(ii)] $V \geq 0$
\item[(iii)] $V$ is even.
\end{enumerate}

\vspace{2mm}

The Hartree equation appears in the mean-field limit of large systems of bosons \cite{Sch,FL}.

\vspace{2mm}

Since the problem  $(\ref{eq:cubicnls})$ is energy-subcritical and defocusing, for fixed $s\geq 1$, (\ref{eq:cubicnls}) has a unique global solution in $H^s$ \cite{Tao}. In this paper, we are interested in estimating $\|u(t)\|_{H^s}$ from above. We recall from \cite{FadTak,ZM} that (\ref{eq:cubicnls}) is completely integrable.
Therefore, if $s=k$ is a positive integer, one can deduce, by using a fixed finite number of conserved quantities that there exists a function $B_k:H^k \rightarrow \mathbb{R}$ such that for all $t \in \mathbb{R}$:

\begin{equation}
\label{eq:Integrals}
\|u(t)\|_{H^k}\leq B_k(\Phi).
\end{equation}

From the preceding observation, it makes sense to consider only the case when \emph{s is not an integer}.
One notes that the uniform bounds for $H^s$ norms when $s$ is not an integer don't follow from the uniform bounds on the integer Sobolev norms if we are assuming only that $\Phi \in H^s(\mathbb{R})$.

\vspace{3mm}

Given a real number $x$, we denote by $x+$ and $x-$ expressions of the form $x+\epsilon$ and $x-\epsilon$ respectively, where $0<\epsilon \ll 1$. With this notation, the result that we prove for (\ref{eq:cubicnls}) is:

\vspace{3mm}

\begin{theorem}(Bound for the Cubic NLS)
\label{Theorem 1}
Suppose $s>1$ is not an integer. Let $\alpha:=s-\lfloor s \rfloor$ denote the fractional part of $s$.
Suppose $\Phi \in H^s(\mathbb{R})$, and let $u$ denote the global solution to the corresponding problem (\ref{eq:cubicnls}).
Then, there exists a continuous function $F_s: H^s \rightarrow \mathbb{R}$ such that for all $t \in \mathbb{R}$:
$$\|u(t)\|_{H^s} \leq F_s(\Phi)(1+|t|)^{\alpha+}.$$
\end{theorem}
Theorem \ref{Theorem 1} gives a solution to an open problem that was mentioned on the Dispersive Wiki Website \cite{DW}.

\vspace{2mm}

Unlike the one-dimensional cubic NLS, the Hartree equation doesn't have infinitely many conserved quantities. The following quantities are conserved under the evolution of $(\ref{eq:Hartree})$:

$$M(u(t))=\int |u(x,t)|^2 dx\,\,\,\mbox{\emph{(Mass)}}$$
and
$$E(u(t))=\frac{1}{2}\int |\nabla u(x,t)|^2 dx + \frac{1}{4}\int (V*|u|^2)(x,t) |u(x,t)|^2 dx \,\,\,\mbox{\emph{(Energy)}}$$
We hence deduce that $\|u(t)\|_{H^1}$ is uniformly bounded whenever $u$ is a solution of $(\ref{eq:Hartree})$.
The bound that we prove is:

\begin{theorem}(Bound for the Hartree equation)
\label{Theorem 2}
Let $s\geq 1$, and let $u$ be the global solution of $(\ref{eq:Hartree})$. Then, there exists a function $C_s$, continuous on $H^1$ such that for all $t \in \mathbb{R}:$

\begin{equation}
\label{eq:hartreebound}
\|u(t)\|_{H^s}\leq
C_s(\Phi)(1+|t|)^{\frac{1}{3}s+}\|\Phi\|_{H^s}.
\end{equation}
\end{theorem}

The bound in Theorem \ref{Theorem 2} is better than the bound $C(1+|t|)^{\frac{1}{2}s+}\|\Phi\|_{H^s}$, which we obtained in the periodic setting in \cite{SoSt1}.

\vspace{2mm}

\begin{remark}
\label{Remark 1.1}
As in \cite{SoSt1}, we can see that the focusing-type analogues of Theorem \ref{Theorem 1} and Theorem \ref{Theorem 2} hold, if we suppose that the initial data is sufficiently small in $L^2$. Namely, if we take $\|\Phi\|_{L^2}$ sufficiently small, Theorem \ref{Theorem 1} holds for the focusing NLS on $\mathbb{R}$. The continuity of the higher conserved quantities is the same \cite{FadTak}. Furthermore, under the same smallness assumption, Theorem \ref{Theorem 2} still holds for $(\ref{eq:Hartree})$ when the convolution potential is not necessarily non-negative, but is still real-valued.
\end{remark}

\subsection{Motivation for the problem and previously known results:}

The growth of Sobolev norms has a physical interpretation in the context of the \emph{Low-to-High frequency cascade}. Namely, we see that $\|u(t)\|_{H^s}$ weighs the higher frequencies more as $s$ becomes larger, and hence its growth gives us a quantitative estimate for how much of the support of $|\widehat{u}|^2$ has transferred from the low to the high frequencies\footnote{We observe that, from conservation of energy, not all of the support of $\widehat{u}$ can move to the high frequencies. If a low-to-high frequency cascade occurs, then a part of $\widehat{u}$ must concentrate near the low frequencies, to counterbalance a movement of $\widehat{u}$ towards the high frequencies. The growth of high Sobolev norms quantitatively describes the latter part of the process.}. This sort of problem also goes under the name \emph{weak turbulence} \cite{BN,BS,Zak}.

\vspace{3mm}

From local well-posedness theory \cite{B3,Ca,Tao}, we know that there exist $C,\tau_0>0$, depending only on the initial data $\Phi$ such that for all $t$:

\begin{equation}
\label{eq:ExponentialIteration}
\|u(t+\tau_0)\|_{H^s}\leq C\|u(t)\|_{H^s}.
\end{equation}
Iterating (\ref{eq:ExponentialIteration}) yields the exponential bound:
\begin{equation}
\label{eq:ExponentialBound}
\|u(t)\|_{H^s}\leq C_1 e^{C_2 t}.
\end{equation}
Here, $C_1,C_2>0$ again depend only on $\Phi$.

\vspace{3mm}

For a wide class of nonlinear dispersive equations, the bound(\ref{eq:ExponentialBound}) can be improved to a polynomial bound, as long as we take $s$ to be an integer, or if we consider sufficiently smooth initial data\footnote{i.e. if we take $\Phi \in H^{\lceil s \rceil}$; This of course only makes sense for the equations which are not completely integrable.}.
This observation was first made in the work of Bourgain \cite{B2}, and continued in the work of Staffilani \cite{S,S2}.

\vspace{3mm}

The crucial step in the mentioned works was to improve the iteration bound (\ref{eq:ExponentialIteration}) to:

\begin{equation}
\label{eq:PolynomialIteration}
\|u(t+\tau_0)\|_{H^s}\leq \|u(t)\|_{H^s} + C \|u(t)\|_{H^s}^{1-r}.
\end{equation}

\vspace{2mm}

As before, $C,\tau_0>0$ depend only on $\Phi$. In this bound, $r \in (0,1)$ satisfies $r \sim \frac{1}{s}$.
One can show that (\ref{eq:PolynomialIteration}) implies that for all $t \in \mathbb{R}$:

\begin{equation}
\label{eq:PolynomialBound}
\|u(t)\|_{H^s}\leq C(\Phi) (1+|t|)^{\frac{1}{r}}.
\end{equation}

\vspace{3mm}

In \cite{B2}, (\ref{eq:PolynomialIteration}) was obtained by using the \emph{Fourier multiplier method}. In \cite{S,S2}, the iteration bound was obtained by using multilinear estimates in $X^{s,b}$-spaces. Similar estimates were used in \cite{KPV3} in the study of well-posedness theory. The key was to use a multilinear estimate in an $X^{s,b}$-space with negative first index $s$. Such a bound was then used as a smoothing estimate. A slightly different approach, based on the analysis of Burq, G\'{e}rard, and Tzvetkov \cite{BGT} is used to obtain (\ref{eq:PolynomialIteration}) in the context of compact Riemannian manifolds in \cite{CatW,Z}.

\vspace{3mm}

An alternative iteration bound, based on the use of the \emph{upside-down I-method}, was used in \cite{SoSt1}, and it gave better polynomial bounds for solutions of nonlinear Schr\"{o}dinger equations on $S^1$. The main idea was to consider the operator $\mathcal{D}$, related to $D^s$ such that $\|\mathcal{D}u\|_{L^2}^2$ is \emph{slowly varying}. A similar technique can be applied to the Hartree equation on $\mathbb{T}^2$ and on $\mathbb{R}^2$. The latter results will be presented in our forthcoming paper \cite{SoSt2}.

\vspace{3mm}

In the paper \cite{B4}, improved polynomial bounds were obtained for the defocusing quintic NLS on $S^1$. The techniques used in this work were based on dynamical systems and Birkhoff normal forms, by which the nonlinearity was reduced to its ``essential part''. The bound given in \cite{B4} for the quintic equation is stronger than the one we obtained in \cite{SoSt1}, but the proof for the stronger bound doesn't seem to work for higher nonlinearities. However, the method given in \cite{SoSt1} works for all nonlinearities.

\vspace{3mm}

All the polynomial bounds mentioned so far involve powers which are essentially a multiple of $s$.
On the other hand, let us consider the linear Schr\"{o}dinger equation on $S^1$ with a real time-dependent potential, i.e.

\begin{equation}
\label{eq:LinearSchrodinger}
iu_t + \Delta u= Vu.
\end{equation}

Here, $V:S^1 \times \mathbb{R} \rightarrow \mathbb{R}$.
If $V$ is taken to be smooth in $x$ and $t$, and we assume that it satisfies the bounds:

\begin{equation}
\label{eq:VBounds}
|\partial_x^{\alpha} \partial_t^{\beta}V|\leq C(\alpha,\beta),\,\,\mbox{for all non-negative integers}\,\,\alpha,\beta.
\end{equation}

Under the assumption (\ref{eq:VBounds}), it is shown in \cite{B5} that, for every $\epsilon>0$, the global solution $u$ of (\ref{eq:LinearSchrodinger}) with initial data $\Phi \in H^s$ satisfies for all $t \in \mathbb{R}$:

\begin{equation}
\label{eq:epsilonBound}
\|u(t)\|_{H^s}\leq C(\Phi,\epsilon)(1+|t|)^{\epsilon}
\end{equation}

The same bound is also proved on $\mathbb{T}^d$, for $d \geq 2$. The proof of the latter result relies on more sophisticated number theoretic arguments.
Furthermore, it was noted in \cite{B6,W} that one obtains an improved logarithmic bound if further regularity assumptions on $V$ are added. Finally, let us note that recently, a new proof of $(\ref{eq:epsilonBound})$ was given in \cite{De}. The argument given in this paper is based on an iterative change of variable. In addition to recovering the result $(\ref{eq:epsilonBound})$ on any $d$-dimensional torus, the same bound is proved for the linear Schr\"{o}dinger equation on any Zoll manifold, i.e. on any compact manifold whose geodesic flow is periodic. It is still an open problem to adapt any of these techniques to obtain improved bounds for nonlinear equations in the periodic case.

\vspace{3mm}

To the best of our knowledge, there are no polynomial bounds not involving powers of $s$ in the non-periodic case except Theorem \ref{Theorem 1}. As we will see, the proof of Theorem \ref{Theorem 1} also works for integer $s$. Hence, the obtained bounds allow us to recover the uniform bounds (\ref{eq:Integrals}) for integer $s$, up to a $t^{0+}$ loss.

\vspace{3mm}

Let us finally mention that the problem of Sobolev norm growth was also studied in a recent paper by Colliander, Keel, Staffilani, Takaoka, and Tao \cite{CKSTT6}, but in the sense of bounding the growth from below. In this paper, the authors exhibit the existence of smooth solutions of the cubic defocusing nonlinear Schr\"{o}dinger equation on $\mathbb{T}^2$, whose $H^s$ norm is arbitrarily small at time zero, and is arbitrarily large at some large finite time. One should note that behavior at infinity is still an open problem.

\subsection{Main ideas of the proofs:}

\subsubsection{Main ideas of the proof of Theorem \ref{Theorem 1}:}

\vspace{1mm}

The main idea of the proof of Theorem \ref{Theorem 1} is to look at the high and low-frequency part of the solution $u$ as in \cite{B5}, and to use the bound (\ref{eq:Integrals}), which gives us uniform bounds on integral Sobolev norms of $u$. In particular, we let $N$ be a parameter, which will be the threshold dividing the ``low'' and ``high'' frequencies, and we define $Q$ to be the projection operator onto the high frequencies. From (\ref{eq:Integrals}), i.e. from the uniform boundedness of the $H^{\lfloor s \rfloor}$ of a solution, we can derive that for all times $t$:

\begin{equation}
\label{eq:(I-Q)u}
\|(I-Q)u(t)\|_{H^s}^2\leq B.
\end{equation}

Here $B=C(\Phi)N^{2\alpha}$, where $\alpha:=s - \rfloor s \lfloor \in [0,1)$ is the fractional part of $s$.
We note that the exponent is then in $[0,2)$ and is not a multiple of $s$.
We use the estimate $(\ref{eq:(I-Q)u})$ to bound the low-frequency part of the solution.

\vspace{3mm}

One then has to bound $\|Qu(t)\|_{H^s}$.
For $t_1>0$, we look at the quantity:

$$\|Qu(t_1)\|_{H^s}^2-\|Qu(t_0)\|_{H^s}^2=\int_{t_0}^{t_1} \frac{d}{dt}\|Qu(t)\|_{H^s}^2 dt.$$

Since we are working on the real line, we can use an appropriate dyadic decomposition and the \emph{improved Strichartz estimate} (Proposition \ref{Proposition 2.3}) to obtain a decay factor of $\frac{1}{N^{1-}}$ in the above integral in time. The exact bounds we obtain are the content of Proposition \ref{Proposition 3.4}. At the end, we deduce that there exists an increment  $\delta>0$, and $C>0$, both depending only on the initial data such that for all $t_0 \in \mathbb{R}$, one has:

\begin{equation}
\label{eq:IncrementIdea}
\|Qu(t_0+\delta)\|_{H^s}^2\leq (1+\frac{C}{N^{1-}})\|Qu(t_0)\|_{H^s}^2 + B_1.
\end{equation}
Here, $B_1\lesssim \frac{1}{N^{1-}}B$.

\vspace{2mm}

The idea now is to iterate (\ref{eq:IncrementIdea}) for times $t_0=0, \delta, \ldots, n\delta,$ where $n \in \mathbb{N}$ is an integer such that $n \lesssim N^{1-}$.

\vspace{2mm}

Multiplying the obtained inequalities by appropriate powers of $1+\frac{C}{N^{1-}}$, and telescoping, we show that:

\begin{equation}
\label{eq:Qundelta}
\|Qu(n\delta)\|_{H^s}^2 \lesssim (1+\frac{C}{N^{1-}})^n \|Qu(0)\|_{H^s}^2 + B
\end{equation}

Since $n \lesssim N^{1-}$, we know:

$$(1+\frac{1}{N^{1-}})^n=O(1).$$

Using the previous bound, $(\ref{eq:(I-Q)u})$ and $(\ref{eq:Qundelta})$, we can show that for all $t \in [0,n \delta]$:

$$\|u(t)\|_{H^s}^2 \lesssim C\|\Phi\|_{H^s}^2 + B.$$

\vspace{2mm}

Optimizing $N$ in terms of the length of time interval $[0,T]$ on which we are considering the solution, and noting that then $B$ becomes the leading term, Theorem \ref{Theorem 1} follows.

\vspace{3mm}

\subsubsection{Main ideas of the proof of Theorem \ref{Theorem 2}:}

\vspace{2mm}

The main argument is similar to the one given in \cite{SoSt1}. Given a parameter $N>1$, we will use the method of an \emph{upside down I-operator}, followed by the method of \emph{higher modified energies} to define a quantity $E^2(u(t))$, which is linked to $\|u(t)\|_{H^s}^2$.

\vspace{2mm}

As in \cite{SoSt1}, our goal is to prove an iteration bound of the type:

    \begin{equation}
    \label{eq:boundforE^2}
    E^2(u(t_0+\delta))\leq (1+\frac{C}{N^{\alpha}})E^2(u(t_0)).
    \end{equation}

for all $t_0 \in \mathbb{R}$, with $\delta,\alpha>0$, and the implied constant all independent of $t_0$.

\vspace{2mm}

Due to the presence of the decay factor $\frac{1}{N^{\alpha}}$, (\ref{eq:boundforE^2}) can be iterated $\sim N^{\alpha}$ times to obtain that $E^2\lesssim 1$ on a time interval of size $\sim N^{\alpha}$.
One then uses the relation between $E^2(u(t))$ and $\|u(t)\|_{H^s}$ to get polynomial bounds for $\|u(t)\|_{H^s}.$

\vspace{2mm}

The bound (\ref{eq:boundforE^2}) is proved in a similar way as the corresponding estimate in \cite{SoSt1}.
In order to construct $E^2$, we need to consider the multiplier $\psi$ which is defined by:
$$\psi:=\frac{(\theta(\xi_1))^2-(\theta(\xi_2))^2+(\theta(\xi_3))^2-(\theta(\xi_4))^2  \widehat{V}(\xi_3+\xi_4)}
{\xi_1^2-\xi_2^2+\xi_3^2-\xi_4^2}$$
when the denominator doesn't vanish, and $\psi:=0$ otherwise. Here, $\theta$ is an appropriately smoothed out and rescaled version of the operator $D^s$.
For details, see $(\ref{eq:theta}),(\ref{eq:definitionofM4})$, and $(\ref{eq:definitionofpsi})$. The key is then to obtain pointwise bounds on such a $\psi$. This is done in Proposition \ref{Proposition 4.2}

\vspace{2mm}

We observe that the bound we obtain in Theorem \ref{Theorem 2} is better than the corresponding bound in the periodic setting. This is a manifestation of stronger dispersion, which is present on the real line. In this paper, we will prove that on $\mathbb{R}$, $(\ref{eq:boundforE^2})$ holds for $\alpha=3-$. We recall from \cite{SoSt1} that the analogous estimate on $S^1$ holds for $\alpha=2-$. Heuristically, the improvement is obtained by using the \emph{improved Strichartz estimate}, which holds on the real line.

\vspace{2mm}

Let us note that Theorem \ref{Theorem 2} would follow trivially if we knew that $(\ref{eq:Hartree})$ scattered in $H^s$, since then all the Sobolev norms of solutions would be uniformly bounded in time. The currently known techniques to prove scattering don't seem to apply in this context though. Namely, the techniques from \cite{GiOz,HNO} require for us to the have additional assumption that our solutions lie in weighted Sobolev spaces, and the obtained bounds depend on these weighted Sobolev norms. Hence we can't argue by density here. The methods from \cite{GiVe} require the initial data to belong to an appropriate subset of the Gevrey class. Finally, the techniques used in \cite{MXZ1,MXZ2} apply only in dimensions greater than or equal to $5$.

\begin{remark}
\label{Remark 1.2}
The techniques of proof of Theorem \ref{Theorem 2} apply to the derivative nonlinear Schr\"{o}dinger equation:

\begin{equation}
\label{eq:dnls}
\begin{cases}
i u_t + \Delta u=i \partial_x (|u|^2u),\\
u(x,0)=\Phi(x),x \in \mathbb{R},\, t\, \in \mathbb{R}.
\end{cases}
\end{equation}

The equation (\ref{eq:dnls}) occurs as a model for the propagation of circularly polarized Alfv\'{e}n
waves in magnetized plasma with a constant magnetic field \cite{SuSu}. In order to obtain global well-posedness in $H^s$, we need to have the smallness assumption:

\begin{equation}
\label{eq:smallnessassumption}
\|\Phi\|_{L^2}<\sqrt{2\pi},
\end{equation}

From \cite{KN}, we know that (\ref{eq:dnls}) is completely integrable. Hence, as for the cubic NLS, it makes sense to bound only the non-integral Sobolev norms of a solution.

\newtheorem*{DNLSbound}{Bound for the Derivative NLS}

\begin{DNLSbound}
For $s>1$, not an integer, and $\Phi \in H^s(\mathbb{R})$, satisfying the smallness assumption
(\ref{eq:smallnessassumption}), there exists $C(s,|\Phi\|_{H^1})$ such that the solution $u$ of
(\ref{eq:dnls}) satisfies:
\begin{equation}
\label{eq:dnlsbound}
\|u(t)\|_{H^s}\leq C(1+|t|)^{2s+} \|\Phi\|_{H^s},\,\mbox{for all}\,t\, \in \mathbb{R}.
\end{equation}
\end{DNLSbound}

The proof of $(\ref{eq:dnlsbound})$ is quite involved. Unlike Theorem \ref{Theorem 1}, we are not able to recover uniform bounds on the integral Sobolev norms of a solution. The techniques that we applied to the cubic NLS don't seem to work for the derivative NLS due to the derivative in the nonlinearity. A sketch of the proof of $(\ref{eq:dnlsbound})$ is given in Appendix C.
\end{remark}

\textbf{Organization of the paper:}

\vspace{2mm}

In Section 2, we give some notation and recall some known facts from Harmonic Analysis. In Section 3, we Prove Theorem \ref{Theorem 1}. Theorem \ref{Theorem 2} is proved in Section 4. Appendix A contains the proofs of auxiliary results for the cubic NLS, whereas Appendix B contains proofs of auxiliary results for the Hartree equation. In Appendix C, we sketch the proof of the bound for the derivative NLS.

\vspace{3mm}

\textbf{Acknowledgements:}

\vspace{2mm}

The author would like to thank his Advisor, Gigliola Staffilani for suggesting this problem, and for her help and encouragement. He would also like to thank Hans Christianson and Antti Knowles for several useful comments and discussions.

\vspace{3mm}

\section{Notation and known facts.}

In our paper, we denote by $A\lesssim B$ an estimate of the form $A\leq CB.$ for some constant $C>0$.
If $C$ depends on $d$, we also write $A \lesssim_d B$ and $C=C(d)$.
\vspace{2mm}
Let us denote by $\|f\|_{L^p}$ the $L^p(\mathbb{R})$ norm, and we denote by $\|f\|_{L^q_tL^r_x}$ the mixed norm:
$$\|f\|_{L^q_tL^r_x}:=\big(\int (\int |f(x,t)|^r dx)^{\frac{q}{r}} dt \big)^{\frac{1}{q}}.$$
with the usual modifications when $q=\infty.$
\vspace{2mm}
We define the spatial Fourier transform of a function $f \in L^2(\mathbb{R})$ by:
$$\widehat{f(\xi)}:=\int_{\mathbb{R}} f(x)e^{-i x \xi} dx.$$
The spacetime Fourier transform of a function $u \in L^2_{t,x}(\mathbb{R}\times \mathbb{R})$ we define by:
$$\widetilde{u}(\xi,\tau):=\int_{\mathbb{R}}\int_{\mathbb{R}} u(x,t) e^{-i (x \xi + t \tau)} dxdt.$$
Let us take the convention for the Japanese bracket to be:
$$\langle \xi \rangle:=\sqrt{1 + |\xi|^2}.$$
Given $s \in \mathbb{R}$, we define the operator $D^s$ by:
$$\widehat{D^s f(\xi)}:=\langle \xi \rangle^s \widehat{f(\xi)}.$$
Furthermore, we define the operator $\dot{D}^s$ by:
$$\widehat{\dot{D}^s f(\xi)}:=|\xi|^s \widehat{f(\xi)}.$$
Also, we define the Sobolev norm of $f=f(x)$:
$$\|f\|_{H^s}:= \|\langle \xi \rangle^s \hat{f}\|_{L^2},$$
and the corresponding Sobolev space:
$$H^{s}(\mathbb{R}):=\{f:\|f\|_{H^s}:=< \infty\}.$$
Let us also define:
$$H^{\infty}(\mathbb{R}):=\bigcap_{s \in \mathbb{R}} H^s(\mathbb{R}).$$
Furthermore, given $s,b \in \mathbb{R}$, we define the $X^{s,b}$ norm of $u=u(x,t)$:
$$\|u\|_{X^{s,b}}:=\|\langle \xi \rangle^s \langle \tau+\xi^2 \rangle^b \widetilde{u}\|_{L^2_{\tau,\xi}},$$
and the corresponding $X^{s,b}$ space
$$X^{s,b}(\mathbb{R}\times \mathbb{R}):=\{u:\|u\|_{X^{s,b}}< \infty\}.$$
We shall usually write the above spaces just as $H^s$ and $X^{s,b}$.
\vspace{3mm}
On $\mathbb{R}$, we recall the following Strichartz estimate (c.f. \cite{B3,Tao}).
\begin{equation}
\label{eq:L6estimate}
\|f\|_{L^6_{t,x}}\lesssim \|f\|_{X^{0,\frac{1}{2}+}}.
\end{equation}
\vspace{3mm}
Interpolating between (\ref{eq:L6estimate}) and $\|f\|_{L^2_{t,x}}=\|f\|_{X^{0,0}}$, it follows that:
\begin{equation}
\label{eq:L4estimate}
\|f\|_{L^4_{t,x}}\lesssim \|f\|_{X^{0,\frac{3}{8}+}}.
\end{equation}
\vspace{3mm}
From Sobolev embedding, we deduce that:
\begin{equation}
\label{eq:LinftyL2estimate}
\|f\|_{L^{\infty}_tL^2_x}\lesssim \|f\|_{X^{0,\frac{1}{2}+}}.
\end{equation}
and:
\begin{equation}
\label{eq:Linftyestimate}
\|f\|_{L^{\infty}_tL^{\infty}_x}\lesssim \|f\|_{X^{\frac{1}{2}+,\frac{1}{2}+}}.
\end{equation}

\vspace{3mm}

Interpolating between (\ref{eq:L6estimate}) and (\ref{eq:LinftyL2estimate}), we obtain:

\begin{equation}
\label{eq:L8L4estimate}
\|f\|_{L^8_tL^4_x}\lesssim \|f\|_{X^{0,\frac{1}{2}+}}.
\end{equation}
\vspace{3mm}
From \cite{SoSt1}, we recall the following localization bound for $X^{s,b}$ spaces.

\begin{lemma}
\label{Lemma 2.2}
If $b \in (0,\frac{1}{2})$ and $s \in \mathbb{R}$, then, for $c<d$:
\begin{equation}
\label{eq:Xsb_localization}
\|\chi_{[c,d]}(t)f\|_{X^{s,b}}\lesssim \|f\|_{X^{s,b+}}
\end{equation}
where the implicit constant doesn't depend on $u,c,d$.
\end{lemma}

For the proof of Lemma \ref{Lemma 2.2}, we refer the reader to the proof of Lemma 2.1. in Appendix A of \cite{SoSt1}.
We remark that the proof of the Lemma given in the periodic case carries over to the non-periodic case. Let us also note that a similar localization result was also proved in \cite{CKSTT4}, and was stated without proof in \cite{CafE}.

\vspace{3mm}

From \cite{B7,CKSTT}, we recall that on $\mathbb{R}$, the following \emph{improved Strichartz estimate} holds:

\begin{proposition}
\label{Proposition 2.3} (Improved Strichartz Estimate)
Suppose $N_1>0$ and suppose that $f,g \in X^{0,\frac{1}{2}+}(\mathbb{R} \times \mathbb{R})$ are such that for all $t \in \mathbb{R}$:

$$supp\, \hat{f}(t) \subseteq \{|\xi| \sim N_1\},\,supp\, \hat{g}(t) \subseteq \{|\xi|\ll N_1\}.$$
Then, the following bound holds:

$$\|fg\|_{L^2_{t,x}}\lesssim \frac{1}{{N_1}^{\frac{1}{2}}}\|f\|_{X^{0,\frac{1}{2}+}}\|g\|_{0,\frac{1}{2}+}.$$
\end{proposition}

We observe the following consequence of Proposition \ref{Proposition 2.3}:

\begin{corollary}
\label{Corollary 2.4}
For $f,g$ as in Proposition \ref{Proposition 2.3}, one has:
$$\|fg\|_{L^{2+}_tL^2_x}\lesssim \frac{1}{N_1^{\frac{1}{2}-}}\|f\|_{X^{0,\frac{1}{2}+}}\|g\|_{X^{0,\frac{1}{2}+}}.$$
\end{corollary}

Let us prove Corollary \ref{Corollary 2.4}.

\begin{proof}
Let $f,g$ be as in the assumptions of the Lemma. We observe that by H\"{o}lder's inequality:

$$\|fg\|_{L^4_tL^2_x} \leq \|f\|_{L^8_tL^4_x} \|g\|_{L^8_tL^4_x} \lesssim
\|f\|_{X^{0,\frac{1}{2}+}}\|g\|_{X^{0,\frac{1}{2}+}}.$$
The last inequality follows from (\ref{eq:L8L4estimate}).

Given $\epsilon>0$ small, we take:

$$\theta:=\frac{2-\epsilon}{2+\epsilon}=1-.$$
Then $\theta \in [0,1]$ satisfies:

$$\theta \cdot \frac{1}{2} + (1-\theta) \cdot \frac{1}{4}=\frac{1}{2+\epsilon}.$$
By using interpolation and Proposition \ref{Proposition 2.3}, we deduce that:

$$\|fg\|_{L^{2+\epsilon}_tL^2_x}\leq (\|fg\|_{L^2_{t,x}})^{\theta} (\|fg\|_{L^4_tL^2_x})^{1-\theta}\lesssim $$

$$({N}_1^{-\frac{1}{2}}\|f\|_{X^{0,\frac{1}{2}+}}\|g\|_{X^{0,\frac{1}{2}+}})^{\theta} (\|f\|_{X^{0,\frac{1}{2}+}}\|g\|_{X^{0,\frac{1}{2}+}})^{1-\theta} \lesssim
{N}_1^{-\frac{\theta}{2}}\|f\|_{X^{0,\frac{1}{2}+}}\|g\|_{X^{0,\frac{1}{2}+}}.$$
Since $\theta=1-$, Corollary \ref{Corollary 2.4} follows.
\end{proof}

\vspace{2mm}

In our analysis, we will have to work with $\chi=\chi_{[t_0,t_0+\delta]}(t)$, the characteristic function of the time interval $[t_0,t_0+\delta]$. It is difficult to deal with $\chi$ directly, since this function is not smooth, and since its Fourier transform doesn't have a sign. Instead, we will decompose $\chi$ as a sum of two functions which are easier to deal with. This goal will be achieved by using an appropriate approximation to the identity.
We will use the following decomposition, which is originally found in \cite{CKSTT}:

\vspace{3mm}

Given $\phi \in C^{\infty}_0(\mathbb{R})$, such that: $0 \leq \phi \leq 1,\, \int_{\mathbb{R}} \,\phi(t)\, dt =1\,$, and $\lambda>0$, we recall that the \emph{rescaling} $\phi_{\lambda}$ of $\phi$ is defined by:

$$\phi_{\lambda}(t):=\frac{1}{\lambda}\,\phi(\frac{t}{\lambda}).$$
\vspace{2mm}
We observe that such a rescaling preserves the $L^1$ norm:

$$\|\phi_{\lambda}\|_{L^1_t}=\|\phi\|_{L^1_t}.$$
\vspace{3mm}
Having defined the rescaling, we write, for the scale $N_1>1$:

\begin{equation}
\label{eq:chi=a+b}
\chi(t)=a(t)+b(t),\,\, \mbox{for}\,\, a:=\chi * \phi_{N_1^{-1}}.
\end{equation}
In Lemma 8.2. of \cite{CKSTT}, the authors note the following estimate:

\begin{equation}
\label{eq:abound}
\|a(t)f\|_{X^{0,\frac{1}{2}+}}\lesssim {N_1}^{0+} \|f\|_{X^{0,\frac{1}{2}+}}.
\end{equation}
(The implied constant here is independent of $N_1$.)

\vspace{2mm}

On the other hand, for any $M \in (1,+\infty)$, one obtains:

$$\|b\|_{L^M_t}=\|\chi-\chi * \phi_{{N_1}^{-1}}\|_{L^M_t} \leq
\|\chi\|_{L^M_t}+\|\chi*\phi_{{N_1}^{-1}}\|_{L^M_t}$$
which is by Young's inequality:
$$\leq \|\chi\|_{L^M_t}+\|\chi\|_{L^M_t}\|\phi_{{N_1}^{-1}}\|_{L^1_t}=2\|\chi\|_{L^M_t}=C(M,\chi)=C(M,\Phi)$$
To explain the fact that $C(M,\chi)=C(M,\Phi)$, we note that $\chi$ is defined as the characteristic function of an interval of size $\delta$, and $\delta$, in turn, depends only on $\Phi$.

\vspace{2mm}

If we now define:

\begin{equation}
\label{eq:b1}
b_1(t):= \int_{\mathbb{R}}|\hat{b}(\tau)|e^{i t \tau} d \tau.
\end{equation}
Since $M \in (1,\infty)$, we know by the Littlewood-Paley inequality \cite{D} that:
$$\|b_1\|_{L^M_t} \lesssim \|b\|_{L^M_t}$$
Then, the previous bound on $\|b\|_{L^M_t}$ implies:

\begin{equation}
\label{eq:b1bound}
\|b_1\|_{L^M_t} \leq C(M,\Phi).
\end{equation}

\vspace{2mm}

We will frequently use the following modification of Proposition \ref{Proposition 2.3}

\begin{proposition}
\label{chiImprovedStrichartz}(Improved Strichartz Estimate with rough cut-off in time)
Suppose $N_1>0$ and suppose $f,g \in X^{0,\frac{1}{2}+}(\mathbb{R} \times \mathbb{R})$ are such that for all $t \in \mathbb{R}$:
$$supp\, \hat{f}(t) \subseteq \{|\xi| \sim N_1\},\,supp\, \hat{g}(t) \subseteq \{|\xi|\ll N_1\}.$$
Let $f_1,g_1$ be given by:
$$\widetilde{f_1}:=|(\chi f)\,\widetilde{}\,|, \widetilde{g_1}:=|\widetilde{g}\,|.$$
Then one has:
\begin{equation}
\label{eq:ImprovedStrichartzchi}
\|f_1 g_1\|_{L^2_{t,x}} \lesssim \frac{1}{N_1^{\frac{1}{2}-}} \|f\|_{X^{0,\frac{1}{2}+}} 
\|g\|_{X^{0,\frac{1}{2}+}}
\end{equation}
The same bound holds if:
$$\widetilde{f_1}:=|\widetilde{f}\,|,\,\widetilde{g_1}:=|(\chi g)\,\widetilde{}\,|.$$
\end{proposition}

\begin{proof}
Let's consider the case when $\widetilde{f_1}=|(\chi f)\,\widetilde{}\,|, \widetilde{g_1}=|\widetilde{g}\,|.$
With notation as earlier, let $F_1,F_2$ be given by:

$$\widetilde{F_1}:=|(a f)\,\widetilde{}\,|, \widetilde{F_2}:=|(b f)\,\widetilde{}\,|.$$
Then, by the triangle inequality, one has:
$$\widetilde{f_1} \leq \widetilde{F_1} + \widetilde{F_2}.$$
Since $\widetilde{f_1},\widetilde{g_1} \geq 0$, Plancherel's Theorem and duality imply that:
$$\|f_1g_1\|_{L^2_{t,x}} \sim \sup_{\|c\|_{L^2_{\tau,\xi}=1}}
\int_{\tau_1+\tau_2+\tau_3=0} \int_{\xi_1+\xi_2+\xi_3=0} \widetilde{f_1}(\xi_1,\tau_1)
\widetilde{g_1}(\xi_2,\tau_2) |c(\xi_3,\tau_3)| d\xi_j d\tau_j$$
$$\leq \sup_{\|c\|_{L^2_{\tau,\xi}=1}}
\int_{\tau_1+\tau_2+\tau_3=0} \int_{\xi_1+\xi_2+\xi_3=0} \widetilde{F_1}(\xi_1,\tau_1)
\widetilde{g_1}(\xi_2,\tau_2) |c(\xi_3,\tau_3)| d\xi_j d\tau_j\,+$$
$$\sup_{\|c\|_{L^2_{\tau,\xi}=1}}
\int_{\tau_1+\tau_2+\tau_3=0} \int_{\xi_1+\xi_2+\xi_3=0} \widetilde{F_2}(\xi_1,\tau_1)
\widetilde{g_1}(\xi_2,\tau_2) |c(\xi_3,\tau_3)| d\xi_j d\tau_j$$
Since $\widetilde{F_1},\widetilde{F_2},\widetilde{v_1} \geq 0$, it follows that the latter expression is
$\sim \|F_1 g_1\|_{L^2_{t,x}} + \|F_2 g_1\|_{L^2_{t,x}}$.
Hence, it follows that:
$$\|f_1 g_1\|_{L^2_{t,x}} \lesssim \|F_1 g_1\|_{L^2_{t,x}} + \|F_2 g_1\|_{L^2_{t,x}}$$

By Proposition \ref{Proposition 2.3}, by the frequency assumptions on $F_1$ and $v_1$, and by the fact that taking absolute values in the spacetime Fourier transform doesn't change the $X^{s,b}$ norms, we know that:

$$\|F_1 g_1\|_{L^2_{t,x}} \lesssim \frac{1}{N_1^{\frac{1}{2}}} \|af\|_{X^{0,\frac{1}{2}+}} \|g\|_{X^{0,\frac{1}{2}+}}$$
We now use $(\ref{eq:abound})$ to deduce that this expression is:
$$\lesssim \frac{1}{N_1^{\frac{1}{2}}}(N^{0+}\|f\|_{X^{0,\frac{1}{2}+}})\|g\|_{X^{0,\frac{1}{2}+}}$$
This expression is:

\begin{equation}
\label{eq:term1chiImprovedStrichartz}
\lesssim \frac{1}{N_1^{\frac{1}{2}-}}\|f\|_{X^{0,\frac{1}{2}+}}\|g\|_{X^{0,\frac{1}{2}+}}
\end{equation}

On the other hand, let us consider $c \in L^2_{\tau,\xi}$.
With notation as before, one has:

$$\big| \int_{\tau_1+\tau_2=0} \int_{\xi_1+\xi_2=0} (F_2 g_1)\,\widetilde{}\,(\xi_1,\tau_1) c(\xi_2,\tau_2) d\xi_j d\tau_j  \big|$$
$$= \big| \int_{\tau_1+\tau_2+\tau_3=0} \int_{\xi_1+\xi_2+\xi_3=0} |(b f)\,\widetilde{}\,(\xi_1,\tau_1)|\widetilde{g_1}(\xi_2,\tau_2) c(\xi_3,\tau_3) d\xi_j d\tau_j \big|$$

$$\leq \int_{\tau_0+\tau_1+\tau_2+\tau_3=0} \int_{\xi_1+\xi_2+\xi_3=0} |\widehat{b}(\tau_0)||\widetilde{f}(\xi_1,\tau_1)| |\widetilde{g_1}(\xi_2,\tau_2)||c(\xi_3,\tau_3)| d\xi_j d\tau_j:=I$$

We then define the functions $G_j,j=1,\ldots,3$ by:

$$\widetilde{G_1}:=|\widetilde{f}|, \widetilde{G_2}:=|\widetilde{g_1}|, \widetilde{G_3}:=|c|$$
Recalling $(\ref{eq:b1})$, and using Parseval's identity, it follows that:

$$I \lesssim \int_{\mathbb{R} \times \mathbb{R}} b_1(t) G_1(x,t) G_2(x,t) G_3(x,t) dx dt$$
We choose $M \in (1,\infty)$, and $2+$ such that: $\frac{1}{M}+\frac{1}{2+}=\frac{1}{2}$. By an $L^M_t, L^{2+}_tL^2_x, L^2_{t,x}$ H\"{o}lder inequality, we deduce that:

$$I \lesssim \|b_1\|_{L^M_t} \|G_1 G_2\|_{L^{2+}_tL^2_x} \|G_3\|_{L^2_{t,x}}$$
We use $(\ref{eq:b1bound})$, Corollary \ref{Corollary 2.4}, and Plancherel's theorem to deduce that:

$$I \lesssim \frac{1}{N_1^{\frac{1}{2}-}} \|f\|_{X^{0,\frac{1}{2}+}} \|g\|_{X^{0,\frac{1}{2}+}}\|c\|_{L^2_{\tau,\xi}}.$$
By duality and by Plancherel's theorem, it follows that:

\begin{equation}
\label{eq:term2chiImprovedStrichartz}
\|F_2 v_1\|_{L^2_{t,x}} \lesssim
\frac{1}{N_1^{\frac{1}{2}-}}\|f\|_{X^{0,\frac{1}{2}+}}\|g\|_{X^{0,\frac{1}{2}+}}
\end{equation}
The case when $\widetilde{f_1}:=|\widetilde{f}\,|,\,\widetilde{g_1}:=|(\chi g)\,\widetilde{}\,|$ is treated analogously.
The Proposition now follows from $(\ref{eq:term1chiImprovedStrichartz})$ and $(\ref{eq:term2chiImprovedStrichartz})$.

\end{proof}

\vspace{3mm}

Furthermore, given a function $v \in L^2_{t,x}$, and a dyadic integer $N$, we define the function $v_N$ as the function obtained from $v$ by restricting its spacetime Fourier Transform to the region $|\xi| \sim N$. We refer to this procedure as a \emph{dyadic decomposition} or \emph{Littlewood-Paley decomposition}. In particular, we can write each function as a sum of such dyadically localized components:

$$v \sim \sum_{dyadic\,\,N} v_N.$$

\vspace{3mm}

Let us give some useful notation for multilinear expressions, which can also be found in \cite{CKSTT,CKSTT5}. For $n \geq 2$, an even integer, we define the hyperplane:

$$\Gamma_n:=\{(\xi_1,\ldots,\xi_n)\in \mathbb{R}^n: \xi_1+\cdots+ \xi_n=0\},$$
endowed with the measure $\delta(\xi_1+\cdots + \xi_n)$. Given a function $M_n=M_n(\xi_1,\ldots,\xi_n)$ on $\Gamma_n$, i.e. an
\emph{n-multiplier}, one defines the \emph{n-linear functional} $\lambda_n(M_n;f_1,\ldots,f_n)$ by:

$$\lambda_n(M_n;f_1,\ldots,f_n):=\int_{\Gamma_n}M_n(\xi_1,\ldots,\xi_n)\prod_{j=1}^n \widehat{f_j}(\xi_j).$$

As in \cite{CKSTT}, we adopt the notation:

\begin{equation}
\label{eq:lambdan}
\lambda_n(M_n;f):=\lambda_n(M_n;f,\bar{f},\ldots,f,\bar{f}).
\end{equation}

We will also sometimes write $\xi_{ij}$ for $\xi_i+\xi_j$,$\xi_{i-j}$ for $\xi_i-\xi_j$, etc.

\vspace{3mm}

Finally, let us recall the following Calculus fact, which is often referred to as the \emph{Double Mean Value Theorem}:

\begin{proposition}
\label{Proposition 2.5}
Let $f \in C^2(\mathbb{R})$. Suppose that $x, \eta, \mu \in \mathbb{R}$ are such that $|\eta|, |\mu| \ll |x|$.
Then, one has:
\begin{equation}
\label{eq:DoubleMVT}
|f(x+\eta+\mu)-f(x+\eta)-f(x+\mu)+f(x)|\lesssim |\eta||\mu||f''(x)|.
\end{equation}
\end{proposition}
The proof of Proposition \ref{Proposition 2.5} follows from the standard Mean Value Theorem.

\section{The cubic nonlinear Schr\"{o}dinger equation.}

\subsection{Basic facts about the equation:}

The equation (\ref{eq:cubicnls}) has the following conserved quantities:

\begin{equation}
\label{eq:ConservationofMass}
M(u(t)):= \int_{\mathbb{R}} |u(x,t)|^2 dx\,\,\emph{(Mass)}
\end{equation}

\begin{equation}
\label{eq:ConservationofEnergy}
E(u(t)):= \frac{1}{2} \int_{\mathbb{R}} |\nabla u(x,t)|^2 dx + \frac{1}{4}\int_{\mathbb{R}} |u(x,t)|^4 dx\,\,\emph{(Energy)}
\end{equation}

\vspace{3mm}

We observe that $\|u(t)\|_{H^1}$ can be bounded by a continuous function of energy and mass. Energy and mass are in turn continuous on $H^1$ by Sobolev embedding.

\vspace{2mm}

The following local-in-time bound will be useful:

\begin{proposition}
\label{Proposition 3.1}
Suppose that $u$ is a global solution of (\ref{eq:cubicnls}). Then, there exist $\delta=\delta(s,Energy,Mass),C=C(s,Energy,Mass)$ such that, for all $t_0 \in \mathbb{R}$, there exists $v \in X^{s,\frac{1}{2}+}$ satisfying the following properties:

\begin{equation}
\label{eq:Property of v1}
v|_{[t_0,t_0+ \delta]}= u|_{[t_0,t_0+\delta]},
\end{equation}

\begin{equation}
\label{eq:Property of v2}
\|v\|_{X^{s,\frac{1}{2}+}}\leq C \|u(t_0)\|_{H^s},
\end{equation}

\begin{equation}
\label{eq:Property of v3}
\|v\|_{X^{1,\frac{1}{2}+}} \leq C.
\end{equation}
Furthermore, $\delta$ and $C$ can be chosen to depend continuously on energy and mass.

\end{proposition}

\vspace{3mm}

The proof of Proposition \ref{Proposition 3.1} proceeds by an appropriate fixed-point method and is analogous to the proof of Proposition 3.1 in \cite{SoSt1}. Furthermore, from the mentioned proof, it follows that $\delta$ and $C$ depend continuously on energy and mass. For the details, we refer the reader to Appendix A in \cite{SoSt1}.

\vspace{3mm}

In the proof of the fact that $F_s$, as in the statement of Theorem \ref{Theorem 1}, depends continuously on the initial data, w.r.t. the $H^s$ topology, we will use the following:

\begin{proposition}(Continuity of conserved quantities)
\label{Proposition 3.2}
Suppose $n$ is a positive integer. Let $E_n$ denote the conserved quantity of (\ref{eq:cubicnls}), which, together with lower-order conserved quantities, we use to bound the $H^n$ norm of a solution. Then $E_n$ is continuous on $H^n$. Moreover, one can construct a function $B_n: H^n \rightarrow \mathbb{R}$ that satisfies (\ref{eq:Integrals}) and is continuous on $H^n$.
\end{proposition}

\vspace{2mm}

The proof of Proposition \ref{Proposition 3.2} is given in Appendix A.

\vspace{3mm}

Although we are starting with initial data $\Phi$, which we are only assuming belongs to $H^s$, and hence with solutions of (\ref{eq:cubicnls}), which we only know belong to $H^s$, our calculations will require us to work with solutions which have more regularity. Hence, we will have to approximate our solutions to (\ref{eq:cubicnls}) with smooth ones, and argue by density. The density argument is made precise by the following result:

\begin{proposition}
\label{Proposition 3.3}
Suppose $u$ satisfies (\ref{eq:cubicnls}) with initial data $\Phi \in H^s$, and suppose each element of $(u^{(n)})$ satisfies (\ref{eq:cubicnls}) with initial data $\Phi_n$, where $\Phi_n \in \mathcal{S}(\mathbb{R})$ and $\Phi_n \stackrel{H^s}{\longrightarrow} \Phi$. Then, one has for all $t$:

$$u^{(n)}(t)\stackrel{H^s}{\longrightarrow} u(t).$$

\end{proposition}

The proof of Proposition \ref{Proposition 3.3} is analogous to the proof of Proposition 3.4. from \cite{SoSt1}, given in Appendix B of the mentioned paper. The proof is very similar, so it will be omitted. We refer the reader to \cite{SoSt1} for details.

\vspace{3mm}

Proposition \ref{Proposition 3.3} allows us to work with smooth solutions and pass to the limit in the end.
Namely, we note that if we take initial data $\Phi_n$ as earlier, then, by persistence of regularity, $u^{(n)}(t)$ will belong to $H^{\infty}(\mathbb{R})$ for all $t$. If we knew that Theorem \ref{Theorem 1} were true for smooth solutions, we would obtain, for all $n \in \mathbb{N}$, and for all $t \in \mathbb{R}$:

$$\|u^{(n)}(t)\|_{H^s} \leq F_s(\Phi_n) (1+|t|)^{\alpha+}.$$

\vspace{2mm}

By letting $n \rightarrow \infty$, and using Proposition \ref{Proposition 3.2} and the continuity of $F_s$ on $H^s$, it would follow that for all $t \in \mathbb{R}$:

$$\|u(t)\|_{H^s} \leq F_s(\Phi) (1+|t|)^{\alpha+}.$$

We may henceforth work with $\Phi \in \mathcal{S}(\mathbb{R})$, which implies that $u(t) \in H^{\infty}(\mathbb{R})$ for all $t$. The claimed result is then deduced from this special case by the approximation procedure given earlier. We will make the same assumption in our study of the Hartree equation.

\vspace{3mm}

\subsection{An Iteration bound and Proof of Theorem \ref{Theorem 1}:}

Let $u$ denote the unique global solution to (\ref{eq:cubicnls}). From the previous arguments, we know that we can assume WLOG that for all $t \in \mathbb{R},\,u(t) \in H^{\infty}(\mathbb{R}).$
Our aim now is to use uniform bounds on $\|u(t)\|_{H^k}$ coming from (\ref{eq:Integrals}) to deduce bounds on $\|u(t)\|_{H^s}$. The key is to perform a frequency decomposition, similarly as in \cite{B5}.

\vspace{3mm}

Let $N>1$ be a parameter which we will determine later. We define the operator $Q$ by:

\begin{equation}
\label{eq:Qoperator}
\widehat{Qf}(\xi):=\chi_{|\xi|\geq N} \hat{f}(\xi).
\end{equation}

\vspace{3mm}

We write $s=k+\alpha$, for $k \in \mathbb{N},\,\alpha \in (0,1).$
Using the definition (\ref{eq:Qoperator}) and (\ref{eq:Integrals}), it follows that:

$$\|(I-Q)u(t)\|_{H^s} \lesssim N^{\alpha} \|(I-Q)u(t)\|_{H^k} \lesssim$$

\begin{equation}
\label{eq:LowFrequencyBound}
\lesssim N^{\alpha}\|u(t)\|_{H^k}  \leq N^{\alpha} B_k(\Phi).
\end{equation}

\vspace{3mm}

We will use (\ref{eq:LowFrequencyBound}) to estimate the \emph{low-frequency part} of the solution.

\vspace{3mm}

The key now is to estimate the \emph{high-frequency part} of the solution. This is done by the following
iteration bound:

\begin{proposition}
\label{Proposition 3.4}
Let $\delta=\delta(\Phi)>0$ be as in Proposition \ref{Proposition 3.1}. Then, there exists a continuous function $C:H^1 \rightarrow \mathbb{R}$ such that for all $t_0 \in \mathbb{R}$, one has:

$$\|Qu(t_0+\delta)\|_{H^s}^2-\|Qu(t_0)\|_{H^s}^2 \leq \frac{C(\Phi)}{N^{1-}} \|u(t_0)\|_{H^s}^2.$$
\end{proposition}

\vspace{3mm}

Before we prove Proposition \ref{Proposition 3.4}, let us note how it implies Theorem \ref{Theorem 1}.

\vspace{3mm}

\begin{proof} (of Theorem \ref{Theorem 1} assuming Proposition \ref{Proposition 3.4})

Let us fix $t_0 \in \mathbb{R}$. It follows that:

$$\|Qu(t_0+\delta)\|_{H^s}^2 \leq \big(1+\frac{C(\Phi)}{N^{1-}} \big) \|Qu(t_0)\|_{H^s}^2 +
\frac{C(\Phi)}{N^{1-}}\|(I-Q)u(t_0)\|_{H^s}^2.$$
By (\ref{eq:LowFrequencyBound}), it follows that:

\begin{equation}
\label{eq:definitionofK}
\frac{C(\Phi)}{N^{1-}}\|(I-Q)u(t_0)\|_{H^s}^2 \lesssim \frac{C(\Phi)}{N^{1-}}N^{2\alpha}B_k^2(\Phi)=:K(N,\Phi).
\end{equation}
\vspace{3mm}
If we multiply $K$ by an appropriate constant, we can write, for all $t_0 \in \mathbb{R}$:

\begin{equation}
\label{eq:IterationBound}
\|Qu(t_0+\delta)\|_{H^s}^2 \leq \big(1+\frac{C(\Phi)}{N^{1-}} \big) \|Qu(t_0)\|_{H^s}^2 + K(N,\Phi).
\end{equation}
Given $n \in \mathbb{N}$, we take $t_0=0, \delta, 2\delta, \ldots, n\delta$ and apply (\ref{eq:IterationBound}) to deduce the inequalities:

\vspace{4mm}

\begin{eqnarray*}
\|Qu(\delta)\|_{H^s}^2 &\leq& \big(1+\frac{C(\Phi)}{N^{1-}} \big) \|Qu(0)\|_{H^s}^2 + K(N,\Phi)\\
\|Qu(2\delta)\|_{H^s}^2 &\leq& \big(1+\frac{C(\Phi)}{N^{1-}} \big) \|Qu(\delta)\|_{H^s}^2 + K(N,\Phi)\\
& \vdots &\\
\|Qu((n-1)\delta)\|_{H^s}^2 &\leq& \big(1+\frac{C(\Phi)}{N^{1-}} \big) \|Qu((n-2)\delta)\|_{H^s}^2 + K(N,\Phi)\\
\|Qu(n\delta)\|_{H^s}^2 &\leq& \big(1+\frac{C(\Phi)}{N^{1-}} \big) \|Qu((n-1)\delta)\|_{H^s}^2 + K(N,\Phi)
\end{eqnarray*}

\vspace{3mm}

Let $\gamma:=1+\frac{C(\Phi)}{N^{1-}}$. Let us multiply the first inequality by $\gamma^{n-1}$, the second inequality by $\gamma^{n-2},\ldots$, and the $(n-1)$-st inequality by $\gamma$. We then sum to obtain:

\begin{equation}
\label{eq:IterationBound2}
\|Qu(n\delta)\|_{H^s}^2 \leq \big(1+\frac{C(\Phi)}{N^{1-}} \big)^n \|Qu(0)\|_{H^s}^2 + K(N,\Phi)(1+ \gamma + \cdots + \gamma^{n-1}).
\end{equation}
\vspace{3mm}
Let us now consider $n$ such that $n \lesssim N^{1-}$.
\vspace{2mm}
For such an $n$, we have:

\begin{equation}
\label{eq:FirstFactor}
(1+\frac{C(\Phi)}{N^{1-}})^n=O(R_1(\Phi)).
\end{equation}
and hence:

$$1+ \gamma+ \cdots + \gamma^{n-1}= \frac{\gamma^n-1}{\gamma-1}= \frac{\big (1+\frac{C(\Phi)}{N^{1-}} \big)^n-1}{\big(1+\frac{C(\Phi)}{N^{1-}} \big)-1}=$$

\begin{equation}
\label{eq:SecondFactor}
=\frac{\big(1+\frac{C(\Phi)}{N^{1-}} \big)^n-1}{\frac{C(\Phi)}{N^{1-}}}=O(N^{1-}R_2(\Phi)).
\end{equation}
We can take the functions $R_1,R_2:H^1 \rightarrow \mathbb{R}$ to be continuous.
\vspace{3mm}
If we then combine (\ref{eq:definitionofK}),(\ref{eq:FirstFactor}),(\ref{eq:SecondFactor}) with (\ref{eq:IterationBound2}), it follows that:

$$\|Qu(n\delta)\|_{H^s}^2 \lesssim R_1(\Phi)\|Q\Phi\|_{H^s}^2+ R_2(\Phi)N^{2\alpha}B_k^2(\Phi).$$
\vspace{3mm}
Hence, by continuity properties of $B_k$ coming from from Proposition \ref{Proposition 3.2}, and by the construction of $R_1,R_2$, we can find a continuous function $R_3:H^s \rightarrow \mathbb{R}$ such that for all $n\lesssim N^{1-}$:

\begin{equation}
\label{eq:HighFrequencyBound}
\|Qu(n\delta)\|_{H^s} \leq R_3(\Phi)(1+N^{\alpha}).
\end{equation}
\vspace{3mm}
Combining (\ref{eq:LowFrequencyBound}) and (\ref{eq:HighFrequencyBound}), we deduce that there exists a continuous function $R_4: H^s \rightarrow \mathbb{R}$, such that for all $n\lesssim N$, one has:

$$\|u(n\delta)\|_{H^s}\leq R_4(\Phi)(1+N^{\alpha}).$$

Finally, by using appropriate local-in-time bounds on each of the $n$ intervals of size $\delta$, it follows that there exists a continuous function $R:H^s \rightarrow \mathbb{R}$ such that, for all $T \lesssim N^{1-}\delta$, one has:

\begin{equation}
\label{eq:HsBound1}
\|u(T)\|_{H^s}\leq R(\Phi)(1+N)^{\alpha}.
\end{equation}

Let us now take:

$$T \sim N^{1-}\delta.$$
Then:

$$N \sim \Big(\frac{T}{\delta}\Big)+.$$

(This is the step in which we choose the parameter $N$.)

\vspace{3mm}

Consequently, since $\delta=\delta(\Phi)>0$ is a continuous function on $H^1$, it follows that there exists a continuous function $F_s$ on $H^s$ such that for $T>1$:

\begin{equation}
\label{eq:HsBound2}
\|u(T)\|_{H^s}\leq F_s(\Phi)(1+T)^{\alpha+}.
\end{equation}

\vspace{3mm}

From local well-posedness, we get the same bound for times in $[0,1]$. By time reversibility, we also get the bound for negative times. Theorem \ref{Theorem 1} now follows.

\end{proof}

\vspace{4mm}
Let us now prove Proposition \ref{Proposition 3.4}

\begin{proof}

We know that: $u_t=i \Delta u - i|u|^2u$. Hence, we compute:

$$\frac{d}{dt} \|Qu(t)\|_{H^s}^2= \frac{d}{dt} \langle D^s Qu(t), D^s Qu(t) \rangle= 2Re \langle D^s Qu, D^s Qu_t \rangle= $$

$$=2Re \langle D^s Qu, D^s u_t \rangle=2 Re \langle D^s Qu, iD^s \Delta u \rangle - 2Re \langle D^s Qu, iD^s(|u|^2u) \rangle =$$

\begin{equation}
\label{eq:HighFrequencyDerivative}
=-2\, Im \langle D^s Qu, D^s(|u|^2u) \rangle.
\end{equation}

\vspace{3mm}

We note that in the third equality, we used Parseval's identity and the definition of $Q$ to omit the operator $Q$ in the second factor, and in the fifth equality, we argued similarly and used the fact that
$$\langle D^s Qu, D^s \Delta u \rangle = \langle D^s Qu, D^s \Delta Qu \rangle \in \mathbb{R}. $$

\vspace{2mm}

It is important to remark that this quantity is indeed finite since $u(t) \in H^{\infty}$. This is what allows us to differentiate in time and use the previous formulae.

\vspace{3mm}

Hence, if we fix $t_0 \in \mathbb{R}$, we obtain:

$$\|Qu(t_0+\delta)\|_{H^s}^2-\|Qu(t_0)\|_{H^s}^2= \int_{t_0}^{t_0+\delta} \frac{d}{dt} \|Qu(t)\|_{H^s}^2=$$

$$=- \int_{t_0}^{t_0+\delta} 2 Im \langle D^s Qu, D^s(|u|^2u) \rangle dt.$$

Thus, it suffices to estimate:

$$|\int_{t_0}^{t_0+\delta} \langle D^s Qu, D^s(|u|^2u) \rangle dt\,|.$$

\vspace{3mm}

Let $v$ be the function we obtain by Proposition \ref{Proposition 3.1}, if we are considering the time $t_0$ we fixed earlier. For the $\delta>0$, which we obtain by Proposition \ref{Proposition 3.1}, we denote:

$$\chi(t):=\chi_{[t_0,t_0+\delta]}(t).$$

Then:

$$ \int_{t_0}^{t_0+\delta} \langle D^s Qu, D^s(|u|^2u) \rangle dt=
\int_{t_0}^{t_0+\delta} \langle D^s Qv, D^s(|v|^2v) \rangle dt=$$

$$=\int_{\mathbb{R}} \int_{\mathbb{R}} \chi(t) D^s Qv D^s(\bar{v} v \bar{v}) dx dt.$$

\vspace{3mm}

With notation as in Section 2 for dyadic integers $N_1,N_2,N_3,N_4$, we define:

$$I_{N_1,N_2,N_3,N_4}:=\int_{\mathbb{R}}\int_{\mathbb{R}} \chi(t) D^s Qv_{N_1} D^s(\overline{v_{N_2}}v_{N_3}
\overline{v_{N_4}}) dx dt.$$

\vspace{3mm}

By definition of $Q$ and the fact that $\sum \xi_j=0,|\xi_j| \sim N_j $, we deduce that $I_{N_1,N_2,N_3,N_4}$ is zero unless the following conditions hold:

\begin{equation}
\label{eq:Conditions on Nj1}
N_1 \gtrsim N.
\end{equation}

\begin{equation}
\label{eq:Conditions on Nj2}
max\,\{N_2,N_3,N_4\} \gtrsim N_1.
\end{equation}
\vspace{3mm}
By Parseval's identity, the expression $I_{N_1,N_2,N_3,N_4}$ is:

\vspace{2mm}

$$\sim \int_{\sum \tau_j =0} \int_{\sum \xi_j =0} (\chi(t) D^s Qv_{N_1})\,\widetilde{\,}\,(\xi_1,\tau_1)
\,\langle \xi_2 \rangle^s \,(\overline{v_{N_2}}v_{N_3}
\overline{v_{N_4}})\,\widetilde{\,}\,(\xi_2,\tau_2) \, d\xi_j d \tau_j=$$

$$=\int_{\sum \tau_j =0} \int_{\sum \xi_j =0} (\chi(t) D^s Qv_{N_1})\,\widetilde{\,}\,(\xi_1,\tau_1)\,
\langle \xi_2 + \xi_3 + \xi_4 \rangle^s \, (\overline{v_{N_2}})\,\widetilde{\,}\,(\xi_2,\tau_2)\,
\widetilde{v_{N_3}}(\xi_3,\tau_3) \, (\overline{v_{N_4}})\,\widetilde{\,}\,(\xi_4,\tau_4) \, d\xi_j d \tau_j$$
\vspace{3mm}
So, by the triangle inequality:

\vspace{3mm}

$$|I_{N_1,N_2,N_3,N_4}| \lesssim$$

$$\int_{\sum \tau_j =0} \int_{\sum \xi_j =0} |(\chi(t) D^s Qv_{N_1})\,\widetilde{\,}\,(\xi_1,\tau_1)|\,
\langle \xi_2 + \xi_3 + \xi_4 \rangle^s \,|(\overline{v_{N_2}})\,\widetilde{\,}\,(\xi_2,\tau_2)|\,
|\widetilde{v_{N_3}}(\xi_3,\tau_3)| \, |(\overline{v_{N_4}})\,\widetilde{\,}\,(\xi_4,\tau_4)| \, d\xi_j d \tau_j.$$
\vspace{4mm}
We now use a \emph{``Fractional Leibniz Rule''}, i.e. we note that:

$$\langle \xi_2 + \xi_3 + \xi_4 \rangle^s \lesssim \langle \xi_2 \rangle^s + \langle \xi_3 \rangle^s + \langle \xi_4 \rangle^s.$$
\vspace{3mm}
Hence, by symmetry \footnote{From the argument that follows, we see that the two other terms are estimated analogously. The fact that the $D^s$ falls on a term with or without a complex conjugate doesn't matter in the argument.}, it suffices to estimate:

\vspace{3mm}

$$J_{N_1,N_2,N_3,N_4}:= $$
$$\int_{\sum \tau_j =0} \int_{\sum \xi_j =0} |(\chi(t) D^s Qv_{N_1})\,\widetilde{\,}\,(\xi_1,\tau_1)|
\,(\langle \xi_2 \rangle^s |(\overline{v_{N_2}})\,\widetilde{\,}\,(\xi_2,\tau_2)|)\,
|\widetilde{v_{N_3}}(\xi_3,\tau_3)| \, |(\overline{v_{N_4}})\,\widetilde{\,}\,(\xi_4,\tau_4)| \, d\xi_j d \tau_j \sim$$

$$\sim \int_{\sum \tau_j =0} \int_{\sum \xi_j =0} |(\chi(t) D^s Qv_{N_1})\,\widetilde{\,}\,(\xi_1,\tau_1)|\,
|(\overline{D^s v_{N_2}})\,\widetilde{\,}\,(\xi_2,\tau_2)|\,
|\widetilde{v_{N_3}}(\xi_3,\tau_3)| \, |(\overline{v_{N_4}})\,\widetilde{\,}\,(\xi_4,\tau_4)| \, d\xi_j d \tau_j.$$
\vspace{3mm}
Let us define:

\begin{equation}
\label{eq:F1}
F_1(x,t):= \int_{\mathbb{R}} \int_{\mathbb{R}} |(\chi(t) D^s Qv_{N_1})\,\widetilde{\,}\,(\xi,\tau)| e^{i (x \xi +t \tau)} d \xi d \tau.
\end{equation}

\begin{equation}
\label{eq:F2}
F_2(x,t):= \int_{\mathbb{R}} \int_{\mathbb{R}} |(D^s v_{N_2})\,\widetilde{\,}\,(\xi,\tau)| e^{i (x \xi +t \tau)} d \xi d \tau.
\end{equation}

\begin{equation}
\label{eq:F34}
F_j(x,t):= \int_{\mathbb{R}} \int_{\mathbb{R}} |\widetilde{v_{N_j}} (\xi,\tau)| e^{i (x \xi +t \tau)} d \xi d \tau,\,\,\mbox{for}\,\,j=3,4.
\end{equation}
\vspace{3mm}
Hence, by Parseval's identity, since all the $\widetilde{F_j}$ are real-valued, we obtain:

\begin{equation}
\label{eq:JNjParseval}
J_{N_1,N_2,N_3,N_4} \sim \int_{\mathbb{R}} \int_{\mathbb{R}} F_1 \overline{F_2} F_3 \overline{F_4} dx dt.
\end{equation}

\vspace{3mm}
We consider the following Cases:

\vspace{4mm}

\textbf{Case 1:} $max\,\{N_2,N_3,N_4 \}=N_3$ or $max\,\{N_2,N_3,N_4 \}=N_4$.

\vspace{3mm}

Let us WLOG suppose that $max\,\{N_2,N_3,N_4 \}=N_3$, since the case $max\,\{N_2,N_3,N_4 \}=N_4$ is analogous.
\vspace{3mm}
Here:

\vspace{3mm}

$|J_{N_1,N_2,N_3,N_4}|=J_{N_1,N_2,N_3,N_4}$, which is by $(\ref{eq:JNjParseval})$ and by an $L^4_{t,x}, L^4_{t,x}, L^4_{t,x}, L^4_{t,x}$ H\"{o}lder's inequality:

$$\lesssim \|F_1\|_{L^4_{t,x}} \|\overline{F_2}\|_{L^4_{t,x}} \|F_3\|_{L^4_{t,x}} \|\overline{F_4}\|_{L^4_{t,x}}=$$

$$=\|F_1\|_{L^4_{t,x}} \|F_2\|_{L^4_{t,x}} \|F_3\|_{L^4_{t,x}} \|F_4\|_{L^4_{t,x}},$$
which by using (\ref{eq:L4estimate}) is:

$$\lesssim \|F_1\|_{X^{0,\frac{3}{8}+}} \|F_2\|_{X^{0,\frac{3}{8}+}} \|F_3\|_{X^{0,\frac{3}{8}+}} \|F_4\|_{X^{0,\frac{3}{8}+}} $$
\vspace{2mm}
By definition of the functions $F_j$, and by the fact that taking absolute values in the spacetime Fourier transform doesn't change the $X^{s,b}$ norm, it follows that the previous expression is:

$$\sim \|\chi(t)D^s Qv_{N_1}\|_{X^{0,\frac{3}{8}+}} \|D^s v_{N_2}\|_{X^{0,\frac{3}{8}+}} \|v_{N_3}\|_{X^{0,\frac{3}{8}+}}
\|v_{N_4}\|_{X^{0,\frac{3}{8}+}}$$
From Lemma \ref{Lemma 2.2} and the fact that $\frac{3}{8}+ <\frac{1}{2}$, this expression is:

$$\lesssim \|D^s Qv_{N_1}\|_{X^{0,\frac{3}{8}++}} \|D^s v_{N_2}\|_{X^{0,\frac{3}{8}+}} \|v_{N_3}\|_{X^{0,\frac{3}{8}+}}
\|v_{N_4}\|_{X^{0,\frac{3}{8}+}} \lesssim $$

$$\lesssim \|D^s Qv_{N_1}\|_{X^{0,\frac{3}{8}++}} \|D^s v_{N_2}\|_{X^{0,\frac{3}{8}+}} \frac{1}{N_3} \|v_{N_3}\|_{X^{1,\frac{3}{8}+}}
\|v_{N_4}\|_{X^{0,\frac{3}{8}+}} \lesssim $$

$$\lesssim \frac{1}{N_3} \|v\|_{X^{s,\frac{1}{2}+}}^2 \|v\|_{X^{1,\frac{1}{2}+}}^2 $$
From Proposition \ref{Proposition 3.1}, we bound this by:

\begin{equation}
\label{eq:Case1bound}
\leq \frac{C(\Phi)}{N_3}\|u(t_0)\|_{H^s}^2.
\end{equation}
\vspace{2mm}
Let us observe that in this case, we have:

$$N_3 \geq N_2,N_4 \,\, \mbox{and} \,\, N_3 \gtrsim N_1 \gtrsim N.$$
Hence, we have obtained a favorable decay factor of $\frac{1}{N_3}$.

\vspace{4mm}

\textbf{Case 2:} $max \{N_2,N_3,N_4\} = N_2$.

\vspace{3mm}

\textbf{Subcase 1:} $N_2 \gg N_3,N_4$.

\vspace{3mm}

Since $\sum \xi_j=0$ and $|\xi_j| \sim N_j$, it follows that $N_1 \sim N_2$.
Hence:

\begin{equation}
\label{eq:NjSubcase1}
N_1 \sim N_2 \gtrsim N \,\, \mbox{and} \,\,  N_1 \sim N_2 \gg N_3,N_4.
\end{equation}

In this subcase, we will have to argue a little bit harder. The main tools that we will use will be the \emph{improved Strichartz estimate} Proposition \ref{Proposition 2.3}, and its modification, Proposition \ref{chiImprovedStrichartz}.

We use $(\ref{eq:JNjParseval})$ and an $L^2_{t,x}, L^2_{t,x}$ H\"{o}lder inequality to deduce that:

$$|J_{N_1,N_2,N_3,N_4}| \lesssim \|F_1 F_3\|_{L^2_{t,x}} \|F_2 F_4\|_{L^2_{t,x}}$$
By the assumption on the frequencies, $(\ref{eq:F1})$, $(\ref{eq:F2})$, $(\ref{eq:F34})$, Proposition \ref{chiImprovedStrichartz} and Proposition \ref{Proposition 2.3}, this expression is:

$$\lesssim \big( \frac{1}{N_1^{\frac{1}{2}-}} \|D^s Qv_{N_1}\|_{X^{0,\frac{1}{2}+}} \|v_{N_3}\|_{X^{0,\frac{1}{2}+}} \big) \big( \frac{1}{N_2^{\frac{1}{2}}}
\|D^s v_{N_2}\|_{X^{0,\frac{1}{2}+}} \|v_{N_4}\|_{X^{0,\frac{1}{2}+}} \big)$$
By $(\ref{eq:NjSubcase1})$, this expression is:

$$\lesssim \frac{1}{N_2^{1-}} \|v\|_{X^{s,\frac{1}{2}+}}^2 \|v\|_{X^{1,\frac{1}{2}+}}^2$$
We now use Proposition \ref{Proposition 3.1} to deduce that in Subcase 1, one has:

\begin{equation}
\label{eq:Subcase1}
|J_{N_1,N_2,N_3,N_4}| \leq \frac{C(\Phi)}{N_2^{1-}} \|u(t_0)\|_{H^s}^2.
\end{equation}
\vspace{3mm}
By $(\ref{eq:NjSubcase1})$, we notice that in this Subcase $\frac{1}{N_2^{1-}}$ is again a favorable decay factor
\vspace{4mm}

\textbf{Subcase 2:} $N_2 \sim N_3 \gtrsim N_4$ or $N_2 \sim N_4 \gtrsim N_3$.

\vspace{3mm}

Let us consider WLOG the case when $N_2 \sim N_3 \gtrsim N_4$, since the case $N_2 \sim N_4 \gtrsim N_3$ is analogous.
\vspace{3mm}
By the same argument as in Case 1, it follows that:

$$|J_{N_1,N_2,N_3,N_4|} \lesssim \frac{1}{N_3} \|v\|_{X^{s,\frac{1}{2}+}}^2 \|v\|_{X^{1,\frac{1}{2}+}}^2.$$
\vspace{3mm}
Since $N_3 \sim N_2$, it follows that:

$$|J_{N_1,N_2,N_3,N_4}| \lesssim \frac{1}{N_2} \|v\|_{X^{s,\frac{1}{2}+}}^2 \|v\|_{X^{1,\frac{1}{2}+}}^2 \leq$$

\begin{equation}
\label{eq:Subcase2}
\leq \frac{C(\Phi)}{N_2} \|u(t_0)\|_{H^s}^2.
\end{equation}
\vspace{3mm}
In this Subcase, $\frac{1}{N_2}$ is an acceptable decay factor.

\vspace{3mm}

Combining (\ref{eq:Subcase1}) and (\ref{eq:Subcase2}), it follows that in Case 2, one has the bound:

\begin{equation}
\label{eq:Case2bound}
|J_{N_1,N_2,N_3,N_4}| \leq \frac{C(\Phi)}{N_2^{1-}}\|u(t_0)\|_{H^s}^2.
\end{equation}
\vspace{3mm}
We now combine (\ref{eq:Case1bound}), (\ref{eq:Case2bound}) and sum in the dyadic integers $N_j$, keeping in mind the assumptions (\ref{eq:Conditions on Nj1}), (\ref{eq:Conditions on Nj2}), and the assumptions of each case. It follows that:

$$|\sum_{N_j} J_{N_1,N_2,N_3,N_4}| \leq \frac{C(\Phi)}{N^{1-}} \|u(t_0)\|_{H^s}^2.$$
Hence, by construction of $J_{N_1,N_2,N_3,N_4}$, we deduce:

$$|\sum_{N_j} I_{N_1,N_2,N_3,N_4}| \leq \frac{C(\Phi)}{N^{1-}} \|u(t_0)\|_{H^s}^2.$$
\vspace{3mm}
The fact that $C(\Phi)$ depends continuously on $\Phi$ w.r.t the $H^1$ topology follows from Proposition \ref{Proposition 3.1}, as well as the same continuous dependence of $\delta$, energy, mass, and the uniform bound on the $H^1$ norm of $u$.
\vspace{3mm}
Proposition \ref{Proposition 3.4} now follows.
\end{proof}

\vspace{5mm}

\section{The Hartree equation.}

\subsection{Basic facts about the equation and definition of the $\mathcal{D}$ operator.}

As in the case of the cubic NLS, we will take $\Phi \in \mathcal{S}(\mathbb{R})$ in order to rigorously justify all of our calculations. The general claim follows by density and the Approximation Lemma, i.e. Proposition \ref{Proposition 3.3} applied to $(\ref{eq:Hartree})$.

\vspace{2mm}

The same iteration argument that we used for the cubic equation doesn't work for $(\ref{eq:Hartree})$, since the only conserved quantities that we have at our disposal are mass and energy. We now adapt to the non-periodic setting the \emph{upside-down I-method} approach that we used on $S^1$ in \cite{SoSt1}.

\vspace{2mm}

We first define $\theta_0: \mathbb{R}\rightarrow \mathbb{R}$ by:

\begin{equation}
\label{eq:theta_0}
\theta_0(\xi) :=
\begin{cases}
  |\xi|^s\,,   \mbox{if }|\xi| \geq 2\\
  \,1,\,  \mbox{if } |\xi| \leq 1.
\end{cases}
\end{equation}
\vspace{2mm}
We extend $\theta_0$ for $1\leq |\xi| \leq 2$ such that $\theta_0$ is even, smooth on $\mathbb{R}$, and such that it is non-decreasing on $[0,+\infty)$.
By construction, we then obtain:

\begin{equation}
\label{eq:theta0prime}
|\theta_0'(\xi)|\lesssim \frac{|\theta_0(\xi)|}{|\xi|}
\end{equation}

\begin{equation}
\label{eq:theta0doubleprime}
|\theta_0''(\xi)|\lesssim \frac{|\theta_0(\xi)|}{|\xi|^2}
\end{equation}

\begin{equation}
\label{eq:theta0sum}
\theta_0(x+y) \lesssim \theta_0(x)+\theta_0(y).
\end{equation}
\vspace{2mm}
Suppose now that $N>1$ is given. Then, we define:

$$\theta(\xi):=\theta_0(\frac{\xi}{N}).$$
Hence:

\begin{equation}
\label{eq:theta}
\theta(\xi) :=
\begin{cases}
  \big(\frac{|\xi|}{N}\big)^s\,,   \mbox{if }|\xi| \geq 2N\\
  \,1,\,  \mbox{if } |\xi| \leq N.
\end{cases}
\end{equation}
\vspace{2mm}
From $(\ref{eq:theta0prime})$,$(\ref{eq:theta0doubleprime})$, and $(\ref{eq:theta0sum})$, we obtain:

\begin{equation}
\label{eq:thetaprime}
|\theta'(\xi)|\lesssim \frac{|\theta(\xi)|}{|\xi|}
\end{equation}

\begin{equation}
\label{eq:thetadoubleprime}
|\theta''(\xi)|\lesssim \frac{|\theta(\xi)|}{|\xi|^2}
\end{equation}

\begin{equation}
\label{eq:thetasum}
\theta(x+y) \lesssim \theta(x)+\theta(y).
\end{equation}
\vspace{2mm}
Having defined $\theta$, we define the $\mathcal{D}$-operator by:

\begin{equation}
\label{eq:Doperator}
\widehat{\mathcal{D}f}(\xi):=\theta(\xi) \widehat{f}(\xi).
\end{equation}
\vspace{2mm}
One then has the bound:

\begin{equation}
\label{eq:Doperatorbound}
\|\mathcal{D}f\|_{L^2} \lesssim \|f\|_{H^s} \lesssim N^s \|\mathcal{D}f\|_{L^2}.
\end{equation}
\vspace{2mm}
Let $u$ denote the global solution of $(\ref{eq:Hartree})$.
We then have the following result:

\vspace{2mm}

\begin{proposition}
\label{Proposition 4.1}
Given $t_0 \in \mathbb{R}$, there exists a globally defined function $v:\mathbb{R} \times \mathbb{R} \rightarrow \mathbb{C}$ satisfying the properties:

\begin{equation}
\label{eq:properties of v12}
v|_{[t_0,t_0+\delta]}=u|_{[t_0,t_0+\delta]}.
\end{equation}

\begin{equation}
\label{eq:properties of v22}
\|v\|_{X^{1,\frac{1}{2}+}}\leq C(s,E(\Phi),M(\Phi))
\end{equation}

\begin{equation}
\label{eq:properties of v32}
\|\mathcal{D}v\|_{X^{0,\frac{1}{2}+}}\leq C(s,E(\Phi),M(\Phi)) \|\mathcal{D}u(t_0)\|_{L^2}.
\end{equation}
Moreover, $\delta$ and $C$ can be chosen to depend continuously on the energy and mass.
\end{proposition}

The proof of Proposition \ref{Proposition 4.1} is analogous to the proof of Proposition 3.1. and Proposition 4.1. in \cite{SoSt1}. The point is that all the intermediate estimates that hold in the periodic setting carry over to the non-periodic setting. Since $V \in L^1(\mathbb{R})$, we know that $\widehat{V} \in L^{\infty}(\mathbb{R})$, so one can directly modify the proof for the cubic NLS to the Hartree equation as in \cite{SoSt1}. We omit the details.

\vspace{2mm}

\subsection{An Iteration bound and proof of Theorem \ref{Theorem 2}}

\vspace{2mm}

As in the periodic case, let:

$$E^1(u(t)):=\|\mathcal{D}u(t)\|_{L^2}^2.$$
Then, arguing as in \cite{SoSt1}, we obtain, that for some $c \in \mathbb{R}$:

$$\frac{d}{dt}E^1(u(t))=ci \int_{\xi_1+\xi_2+\xi_3+\xi_4=0}((\theta(\xi_1))^2-(\theta(\xi_2))^2+(\theta(\xi_3))^2-(\theta(\xi_4))^2)$$
\begin{equation}
\label{eq:firstcontribution}
\widehat{V}(\xi_3+\xi_4)\widehat{u}(\xi_1)\widehat{\bar{u}}(\xi_2)\widehat{u}(\xi_3)\widehat{\bar{u}}(\xi_4) d\xi_j.
\end{equation}

\vspace{4mm}

Recalling the notation from the Introduction, as in \cite{SoSt1}, we consider the following \emph{higher modified energy}

\begin{equation}
\label{eq:modifiedenergy}
E^2(u):=E^1(u) + \lambda_4(M_4;u).
\end{equation}

\vspace{3mm}

The quantity $M_4$ will be determined soon.

\vspace{4mm}

The modified energy $E^2$ is obtained by adding a ``multilinear correction'' to the modified energy $E^1$ considered earlier.
In order to find $\frac{d}{dt}E^2(u)$, we need to find $\frac{d}{dt}\lambda_4(M_4;u)$. Thus, if we fix a multiplier $M_4$, we obtain:

$$\frac{d}{dt} \lambda_4(M_4;u)=$$
\vspace{2mm}
$$=-i\lambda_4(M_4(\xi_1^2-\xi_2^2+\xi_3^2-\xi_4^2);u)$$
\vspace{2mm}
$$-i\int_{\xi_1+\xi_2+\xi_3+\xi_4+\xi_5+\xi_6=0} \big[M_4(\xi_{123},\xi_4,\xi_5,\xi_6)\widehat{V}(\xi_1+\xi_2)$$
\vspace{2mm}
$$-M_4(\xi_1,\xi_{234},\xi_5,\xi_6)\widehat{V}(\xi_2+\xi_3)+M_4(\xi_1,\xi_2,\xi_{345},\xi_6)\widehat{V}(\xi_3+\xi_4)$$
\vspace{2mm}
\begin{equation}
\label{eq:secondcontribution}
-M_4(\xi_1,\xi_2,\xi_3,\xi_{456})\widehat{V}(\xi_4+\xi_5)\big]\widehat{u}(\xi_1)\widehat{\bar{u}}(\xi_2)\widehat{u}(\xi_3)
\widehat{\bar{u}}(\xi_4)\widehat{u}(\xi_5)\widehat{\bar{u}}(\xi_6) d \xi_j
\end{equation}

\vspace{2mm}

With the setup $(\ref{eq:firstcontribution})$ and $(\ref{eq:secondcontribution})$, we can use \emph{higher modified energies} as in in the periodic setting. Namely, it follows that if we take:

\begin{equation}
\label{eq:definitionofM4}
M_4:=\Psi.
\end{equation}
where $\Psi$ is defined by:

$$\Psi: \Gamma_4 \rightarrow \mathbb{R}$$

\begin{equation}
\label{eq:definitionofpsi}
\Psi:=
\begin{cases}
  c\frac{(\theta(\xi_1))^2-(\theta(\xi_2))^2+(\theta(\xi_3))^2-(\theta(\xi_4))^2  \widehat{V}(\xi_3+\xi_4)}
  {\xi_1^2-\xi_2^2+\xi_3^2-\xi_4^2},\,   \mbox{if }\xi_1^2-\xi_2^2+\xi_3^2-\xi_4^2 \neq 0 \\
  \,0,\,  \mbox{otherwise. }
\end{cases}
\end{equation}
for an appropriate real constant $c$.
\vspace{4mm}
One then has:

\begin{equation}
\label{eq:incremenentE2}
\frac{d}{dt}E^2(u)=-i \lambda_6(M_6;u).
\end{equation}
\vspace{2mm}
where:

\vspace{2mm}

$$M_6(\xi_1,\xi_2,\xi_3,\xi_4,\xi_5,\xi_6):=
M_4(\xi_{123},\xi_4,\xi_5,\xi_6)\widehat{V}(\xi_1+\xi_2)$$
\vspace{2mm}
$$-M_4(\xi_1,\xi_{234},\xi_5,\xi_6)\widehat{V}(\xi_2+\xi_3)+M_4(\xi_1,\xi_2,\xi_{345},\xi_6)\widehat{V}(\xi_3+\xi_4)$$
\vspace{2mm}
\begin{equation}
\label{eq:M_6_Hartree}
-M_4(\xi_1,\xi_2,\xi_3,\xi_{456})\widehat{V}(\xi_4+\xi_5)
\end{equation}
\vspace{2mm}
The key to continue our study of $E^2(u)$ is to deduce pointwise bounds on $\Psi$.
We dyadically localize the frequencies as $|\xi_j| \sim N_j$.
We then order the $N_j^s$ in decreasing order to obtain:
$N_1^* \geq N_2^* \geq N_3^* \geq N_4^*$.
Let us show that the following result holds:

\begin{proposition}(Pointwise bound on the multiplier)
\label{Proposition 4.2}
Under the previous assumptions, one has:
\begin{equation}
\label{eq:psipointwisebound1}
\mbox{If}\,\,\,\,N_2^*\gg N_3^*,\,\,\Psi=O(\frac{1}{(N_1^*)^2}\theta(N_1^*)\theta(N_2^*)).
\end{equation}
\vspace{2mm}
\begin{equation}
\label{eq:psipointwisebound2}
\mbox{If}\,\,\,\,N_2^* \sim N_3^*,\,\,\Psi=O(\frac{1}{(N_1^*)^3}\theta(N_1^*)\theta(N_2^*)N_3^*N_4^*).
\end{equation}
\end{proposition}

\vspace{2mm}
In the proof of Proposition \ref{Proposition 4.2}, the following bound will be useful:

\begin{lemma}
\label{Lemma 4.3}
Suppose that $|x|\geq |y|$. Then, one has:
$$|(\theta(x))^2-(\theta(y))^2| \lesssim (|x|-|y|)\frac{(\theta(x))^2}{|x|}.$$
\end{lemma}

We prove Proposition \ref{Proposition 4.2} and Lemma \ref{Lemma 4.3} in Appendix B.

\vspace{3mm}

Using Proposition \ref{Proposition 4.2} and arguing as in \cite{SoSt1}, we deduce that, whenever $u$ is a global solution of $(\ref{eq:Hartree})$, one has:

\begin{equation}
\label{eq:E1E2equivalence}
E^2(u) \sim E^1(u).
\end{equation}
\vspace{3mm}
Arguing as in \cite{SoSt1}, the key is to deduce the following bound:

\begin{lemma}
\label{Bigstar}
For all $t_0 \in \mathbb{R}$, one has:
$$|E^2(u(t_0+\delta))-E^2(u(t_0))|\lesssim \frac{1}{N^{3-}}E^2(u(t_0)).$$
\end{lemma}
\vspace{3mm}
We see that Theorem \ref{Theorem 2} follows from Lemma \ref{Bigstar}:

\begin{proof}(of Theorem \ref{Theorem 2} assuming Lemma \ref{Bigstar})

By Lemma \ref{Bigstar}, there exists $C>0$ such that for all $t_0 \in \mathbb{R}$, one has:

\begin{equation}
\label{eq:bigstar}
E^2(u(t_0+\delta))\leq (1+\frac{C}{N^{3-}})E^2(u(t_0))
\end{equation}
Using (\ref{eq:bigstar}) iteratively, we obtain that \footnote{
Strictly speaking, we are using $(\ref{eq:properties of v32})$ to deduce that we can get the bound for all such times, and not just those which are a multiple of $\delta$.} $\forall T>1:$ \,

$$E^2(u(T)) \leq (1+\frac{C}{N^{3-}})^{\lceil \frac{T}{\delta} \rceil} E^2(\Phi) $$
Let us take:

\begin{equation}
\label{eq:Iteration2}
T \sim N^{3-}
\end{equation}
For such a choice of $T$, one has:

\begin{equation}
\label{eq:Iteration3}
E^2(u(T))\lesssim E^2(\Phi)
\end{equation}
Using $(\ref{eq:Doperatorbound})$, and $(\ref{eq:E1E2equivalence})$, it follows that:

$$\|u(T)\|_{H^s} \lesssim N^s E^2(u(T)) \lesssim N^s E^2(\Phi) \lesssim N^s \|\Phi\|_{H^s} $$
\vspace{1mm}
\begin{equation}
\label{eq:QB1}
\lesssim T^{\frac{s}{3}+} \|\Phi\|_{H^s}\lesssim (1+T)^{\frac{s}{3}+} \|\Phi\|_{H^s}.
\end{equation}
Since for times $t\in [0,1]$, we get the bound of Theorem \ref{Theorem 2} just by iterating the local well-posedness construction, the claim for these times follows immediately. Combining this observation, (\ref{eq:QB1}), recalling the approximation result, and using time-reversibility, Theorem \ref{Theorem 2} follows.
\end{proof}

We now prove Lemma \ref{Bigstar}.

\begin{proof}

Let us WLOG consider $t_0=0$. The general case follows analogously. By $(\ref{eq:incremenentE2})$, we write:

$$E^2(u(\delta))-E^2(u(0))= \int_0^{\delta} \frac{d}{dt} E^2(u(t)) dt = - i \int_0^{\delta} \lambda_6(M_6;u) dt$$

\vspace{3mm}

We recall $(\ref{eq:M_6_Hartree})$, and we use symmetry to deduce that it suffices to bound:

\vspace{2mm}

$$\int_0^{\delta} \int_{\xi_1+\cdots+\xi_6=0} M_4(\xi_{123},\xi_4,\xi_5,\xi_6) \widehat{V}(\xi_1+\xi_2)
\widehat{u}(\xi_1) \widehat{\bar{u}}(\xi_2) \widehat{u}(\xi_3) \widehat{\bar{u}}(\xi_4) \widehat{u}(\xi_5)
\widehat{\bar{u}}(\xi_6) d\xi_j dt$$
\vspace{2mm}
Let $v$ be as in Proposition \ref{Proposition 4.1}, and let $\chi=\chi(t)=\chi_{[0,\delta]}(t)$.
The above expression is then equal to:

\vspace{3mm}

$$\int_0^{\delta} \int_{\xi_1+\cdots+\xi_6=0} M_4(\xi_{123},\xi_4,\xi_5,\xi_6) \widehat{V}(\xi_1+\xi_2)
\widehat{v}(\xi_1) \widehat{\bar{v}}(\xi_2) \widehat{v}(\xi_3) \widehat{\bar{v}}(\xi_4) \widehat{v}(\xi_5)
(\chi \bar{v})\,\widehat{}\,(\xi_6) d\xi_j dt=$$
\vspace{2mm}
$$=\int_{\tau_1+\cdots+\tau_4=0} \int_{\xi_1+\cdots+\xi_4=0} M_4(\xi_1,\xi_2,\xi_3,\xi_4)$$
\vspace{2mm}
$$((V*|v|^2)v)\,\widetilde{}\,(\xi_1,\tau_1)\widetilde{\bar{v}}(\xi_2,\tau_2)
\widetilde{v}(\xi_3,\tau_3) (\chi\bar{v})\,\widetilde{}\,(\xi_4,\tau_4) d\xi_j d\tau_j$$
\vspace{3mm}
Let $N_j,j=1,\ldots 4$, be dyadic integers. We define:

$$I_{N_1,N_2,N_3,N_4}:= \int_{\tau_1+\cdots+\tau_4=0} \int_{\xi_1+\cdots+\xi_4=0} M_4(\xi_1,\xi_2,\xi_3,\xi_4)$$
\vspace{2mm}
$$((V*|v|^2)v)\,\widetilde{}_{N_1}\,(\xi_1,\tau_1)\widetilde{\bar{v}}_{N_2}(\xi_2,\tau_2)
\widetilde{v}_{N_3}(\xi_3,\tau_3) (\chi\bar{v})\,\widetilde{}_{N_4}\,(\xi_4,\tau_4) d\xi_j d\tau_j.$$
\vspace{2mm}
We want to bound $I_{N_1,N_2,N_3,N_4}.$
\vspace{3mm}
Let us define by $N_j^*$ the appropriate reordering of the $N_j$. We know:

\begin{equation}
\label{eq:NjstarE2}
N_1^* \sim N_2^*,\,N_1^* \gtrsim N.
\end{equation}
\vspace{2mm}
We have to consider two Big Cases:

\vspace{3mm}

\textbf{Big Case 1:} $N_2^* \gg N_3^*$.

\vspace{2mm}

\textbf{Big Case 2:} $N_2^* \sim N_3^*$.

\vspace{3mm}

\textbf{Big Case 1:}

\vspace{2mm}

From Proposition \ref{Proposition 4.2}, in this Big Case, we have the bound:

\begin{equation}
\label{eq:M4BigCase1}
M_4=O(\frac{1}{(N_1^*)^2}\theta(N_1^*)\theta(N_2^*))
\end{equation}

We consider several Cases:

\vspace{2mm}

\textbf{Case 1:} $N_1^*\sim N_1$ (and hence $N_2^*\sim N_1$).

\vspace{2mm}
Let us assume WLOG that:

$$N_2^* \sim N_2, N_3^* \sim N_3, N_4^* \sim N_4.$$
The other cases are analogous.

\vspace{2mm}

By using $(\ref{eq:VhatHartree})$, we deduce:

$$|I_{N_1,N_2,N_3,N_4}| \leq \int_{\tau_1+\cdots+\tau_6=0} \int_{\xi_1+\cdots+\xi_6=0, |\xi_1+\xi_2+\xi_3|\sim N_1} \frac{1}{(N_1^*)^2} \theta(\xi_1+\xi_2+\xi_3) \theta(N_2^*)$$
\vspace{2mm}
$$|\widetilde{v}(\xi_1,\tau_1)||\widetilde{\bar{v}}(\xi_2,\tau_2)||\widetilde{v}(\xi_3,\tau_3)|
|\widetilde{\bar{v}}_{N_2}(\xi_4,\tau_4)|
|\widetilde{v}_{N_3}(\xi_5,\tau_5)||(\chi\bar{v})\,\widetilde{}_{N_4}\,(\xi_6,\tau_6)| d\xi_j d\tau_j.$$
\vspace{2mm}
From $(\ref{eq:thetasum})$, we know that:

$$\theta(\xi_1+\xi_2+\xi_3)\lesssim \theta(\xi_1)+\theta(\xi_2)+\theta(\xi_3).$$
\vspace{2mm}
By symmetry, we need to bound:

$$I^1_{N_1,N_2,N_3,N_4}:=\int_{\tau_1+\cdots+\tau_6=0} \int_{\xi_1+\cdots+\xi_6=0, |\xi_1+\xi_2+\xi_3|\sim N_1} \frac{1}{(N_1^*)^2} \theta(\xi_1) \theta(N_2^*)$$
\vspace{2mm}
$$|\widetilde{v}(\xi_1,\tau_1)||\widetilde{\bar{v}}(\xi_2,\tau_2)||\widetilde{v}(\xi_3,\tau_3)|
|\widetilde{\bar{v}}_{N_2}(\xi_4,\tau_4)|
|\widetilde{v}_{N_3}(\xi_5,\tau_5)||(\chi\bar{v})\,\widetilde{}_{N_4}\,(\xi_6,\tau_6)| d\xi_j d\tau_j \lesssim$$
\vspace{2mm}
$$\int_{\tau_1+\cdots+\tau_6=0} \int_{\xi_1+\cdots+\xi_6=0, |\xi_1+\xi_2+\xi_3|\sim N_1} \frac{1}{(N_1^*)^2}
|(\mathcal{D}v)\,\widetilde{}\,(\xi_1,\tau_1)|$$
\vspace{2mm}
$$|\widetilde{\bar{v}}(\xi_2,\tau_2)||\widetilde{v}(\xi_3,\tau_3)|
|(\mathcal{D}\bar{v})\,\widetilde{}_{N_2}\,(\xi_4,\tau_4)|
|\widetilde{v}_{N_3}(\xi_5,\tau_5)||(\chi\bar{v})\,\widetilde{}_{N_4}\,(\xi_6,\tau_6)| d\xi_j d\tau_j.$$
\vspace{2mm}
Now, $|\xi_1+\xi_2+\xi_3| \sim N_1$, hence:

$$\max\{|\xi_1|,|\xi_2|,|\xi_3|\} \gtrsim N_1.$$
We have to consider several subcases:

\vspace{3mm}

\textbf{Subcase 1:} $|\xi_1| \gtrsim N_1.$

\vspace{2mm}

The contribution to $I^1_{N_1,N_2,N_3,N_4}$ in this subcase is:

$$\lesssim \int_{\tau_1+\cdots+\tau_6=0} \int_{\xi_1+\cdots+\xi_6=0} \frac{1}{(N_1^*)^2}
\big(|(\mathcal{D}v)\,\widetilde{}\,_{\gtrsim N_1}(\xi_1,\tau_1)||\widetilde{v}_{N_3}(\xi_5,\tau_5)|\big)$$
\vspace{2mm}
$$\big(|(\mathcal{D}\bar{v})\,\widetilde{}_{N_2}\,(\xi_4,\tau_4)|
|(\chi\bar{v})\,\widetilde{}_{N_4}\,(\xi_6,\tau_6)|\big)
|\widetilde{\bar{v}}(\xi_2,\tau_2)||\widetilde{v}(\xi_3,\tau_3)|
d\xi_j d\tau_j=$$
\vspace{2mm}
$$=\int_{\mathbb{R}} \int_{\mathbb{R}} \frac{1}{(N_1^*)^2} F_1 F_2 \overline{F_3} \overline{F_4}
\overline{F_5} F_6 dx dt$$
\vspace{3mm}
For the last equality, we used Parseval's identity for the functions $F_j$, which are chosen to satisfy:

$$\widetilde{F_1}=|(\mathcal{D}v)\,\widetilde{}_{\gtrsim N_1}|,
\widetilde{F_2}=|\widetilde{v}_{N_3}|,\,\widetilde{F_3}=|(\mathcal{D}v)\,\widetilde{}_{N_2}|,
\widetilde{F_4}=|(\chi v)\,\widetilde{}_{N_4}|,\widetilde{F_5}=|\widetilde{v}|,\widetilde{F_6}=|\widetilde{v}|.$$
We now use an $L^2_{t,x}, L^2_{t,x}, L^{\infty}_{t,x}, L^{\infty}_{t,x}$ H\"{o}lder inequality, Proposition \ref{Proposition 2.3}, Proposition \ref{chiImprovedStrichartz}, and $(\ref{eq:Linftyestimate})$ to see that this expression is:

$$\lesssim \frac{1}{(N_1^*)^2} \big(\frac{1}{N_1^{\frac{1}{2}}} \|\mathcal{D}v\|_{X^{0,\frac{1}{2}+}} \|v\|_{X^{0,\frac{1}{2}+}}\big)\big(\frac{1}{N_2^{\frac{1}{2}-}} \|\mathcal{D}v\|_{X^{0,\frac{1}{2}+}} \|v\|_{X^{0,\frac{1}{2}+}}\big) \|v\|_{X^{\frac{1}{2}+,\frac{1}{2}+}}^2$$
\vspace{2mm}
$$\lesssim \frac{1}{(N_1^*)^{3-}} \|\mathcal{D}v\|_{X^{0,\frac{1}{2}+}}^2$$
\vspace{2mm}
\begin{equation}
\label{eq:BigCase1Case1Subcase1}
\lesssim \frac{1}{(N_1^*)^{3-}} E^2(u(0)).
\end{equation}

\vspace{3mm}

\textbf{Subcase 2:} $|\xi_2| \gtrsim N_1.$

\vspace{2mm}

The contribution to $I^1_{N_1,N_2,N_3,N_4}$ in this subcase is:

$$\lesssim \int_{\tau_1+\cdots+\tau_6=0} \int_{\xi_1+\cdots+\xi_6=0} \frac{1}{(N_1^*)^2}
|(\mathcal{D}v)\,\widetilde{}\,(\xi_1,\tau_1)||(\bar{v})\,\widetilde{}\,_{\gtrsim N_1}(\xi_2,\tau_2)|$$
\vspace{2mm}
$$\big(|(\mathcal{D}\bar{v})\,\widetilde{}_{N_2}\,(\xi_4,\tau_4)|
|(\chi\bar{v})\,\widetilde{}_{N_4}\,(\xi_6,\tau_6)|\big)
|\widetilde{v}(\xi_3,\tau_3)||\widetilde{v}_{N_3}(\xi_5,\tau_5)|
d\xi_j d\tau_j$$
\vspace{4mm}

We argue similarly as in the previous Subcase, but we now use an $L^4_{t,x},L^4_{t,x},L^2_{t,x},L^{\infty}_{t,x},L^{\infty}_{t,x}$ H\"{o}lder inequality and Proposition \ref{chiImprovedStrichartz} to deduce that the previous expression is:

$$\lesssim \frac{1}{(N_1^*)^2} \|\mathcal{D}v\|_{X^{0,\frac{3}{8}+}} \|v_{\gtrsim N_1}\|_{X^{0,\frac{3}{8}+}}
\big(\frac{1}{N_2^{\frac{1}{2}-}}\|\mathcal{D}v\|_{X^{0,\frac{1}{2}+}}\|v\|_{X^{0,\frac{1}{2}+}}\big)
\|v\|_{X^{\frac{1}{2}+,\frac{1}{2}+}}^2$$
\vspace{2mm}
$$\lesssim \frac{1}{(N_1^*)^{\frac{7}{2}-}}\|\mathcal{D}v\|_{X^{0,\frac{1}{2}+}}^2\|v\|_{X^{1,\frac{1}{2}+}}^4$$
\vspace{2mm}

\begin{equation}
\label{eq:BigCase1Case1Subcase2}
\lesssim \frac{1}{(N_1^*)^{\frac{7}{2}-}}E^2(u(0)).
\end{equation}

\vspace{2mm}

(We note that we used the fact that $\|v_{\gtrsim N_1}\|_{X^{0,\frac{3}{8}+}}\lesssim \frac{1}{N_1}\|v\|_{X^{1,\frac{1}{2}+}}.$)

\vspace{3mm}

\textbf{Subcase 3:} $|\xi_3|\gtrsim N_1$.

\vspace{2mm}

Subcase 3 is analogous to Subcase 2, and we get the same bound on the wanted contribution.

\vspace{3mm}

\textbf{Case 2:} $N_3^* \sim N_1$ or $N_4^* \sim N_1$.
\vspace{2mm}
Let us WLOG consider the case $N_3^* \sim N_1$. (the case $N_4^* \sim N_1$ is analogous)
\vspace{2mm}
Let us also WLOG suppose:

$$N_1^* \sim N_2,\, N_2^* \sim N_3,\, N_4^* \sim N_4.$$
\vspace{2mm}
Arguing similarly as earlier, we want to estimate:

$$\int_{\tau_1+\cdots+\tau_6=0} \int_{\xi_1+\cdots+\xi_6=0, |\xi_1+\xi_2+\xi_3|\sim N_1} \frac{1}{(N_1^*)^2} \theta(N_2) \theta(N_3)$$
\vspace{2mm}
$$|\widetilde{v}(\xi_1,\tau_1)||\widetilde{\bar{v}}(\xi_2,\tau_2)||\widetilde{v}(\xi_3,\tau_3)|
|\widetilde{\bar{v}}_{N_2}(\xi_4,\tau_4)|
|\widetilde{v}_{N_3}(\xi_5,\tau_5)||(\chi\bar{v})\,\widetilde{}_{N_4}\,(\xi_6,\tau_6)| d\xi_j d\tau_j$$
\vspace{2mm}
We write:

$$v=v_{\ll (N_1^*)^{\frac{1}{2}}}+v_{\gtrsim (N_1^*)^{\frac{1}{2}}}.$$
\vspace{2mm}
We consider the following subcases:

\vspace{3mm}

\textbf{Subcase 1:} $|\xi_1|,|\xi_2|,|\xi_3| \ll (N_1^*)^{\frac{1}{2}}.$

\vspace{2mm}

We have to estimate:

$$\int_{\tau_1+\cdots+\tau_6=0} \int_{\xi_1+\cdots+\xi_6=0, |\xi_1+\xi_2+\xi_3|\sim N_1} \frac{1}{(N_1^*)^2} \theta(N_2) \theta(N_3)$$
\vspace{2mm}
$$|\widetilde{v}_{\ll (N_1^*)^{\frac{1}{2}}}(\xi_1,\tau_1)||\widetilde{\bar{v}}_{\ll (N_1^*)^{\frac{1}{2}}}(\xi_2,\tau_2)||\widetilde{v}_{\ll (N_1^*)^{\frac{1}{2}}}(\xi_3,\tau_3)|
|\widetilde{\bar{v}}_{N_2}(\xi_4,\tau_4)|
|\widetilde{v}_{N_3}(\xi_5,\tau_5)||(\chi\bar{v})\,\widetilde{}_{N_4}\,(\xi_6,\tau_6)| d\xi_j d\tau_j$$
\vspace{2mm}
$$\lesssim \int_{\tau_1+\cdots+\tau_6=0} \int_{\xi_1+\cdots+\xi_6=0, |\xi_1+\xi_2+\xi_3|\sim N_1} \frac{1}{(N_1^*)^2}
(|\widetilde{v}_{\ll (N_1^*)^{\frac{1}{2}}}(\xi_1,\tau_1)||(\mathcal{D}\bar{v})\,\widetilde{}\,_{N_2}(\xi_4,\tau_4)|)$$
\vspace{2mm}
$$(|(\mathcal{D}v)\,\widetilde{}\,_{N_3}(\xi_5,\tau_5)||(\chi\bar{v})\,\widetilde{}_{N_4}\,(\xi_6,\tau_6)|)
|\widetilde{v}_{\ll (N_1^*)^{\frac{1}{2}}}(\xi_3,\tau_3)||\widetilde{v}_{\ll (N_1^*)^{\frac{1}{2}}}(\xi_2,\tau_2)|
d\xi_j d\tau_j$$
\vspace{3mm}

We apply an $L^2_{t,x},L^2_{t,x},L^{\infty}_{t,x},L^{\infty}_{t,x}$ H\"{o}lder inequality, Proposition \ref{Proposition 2.3}, and Proposition \ref{chiImprovedStrichartz} to deduce that the above expression is:

$$\lesssim \frac{1}{(N_1^*)^2} \big(\frac{1}{N_2^{\frac{1}{2}}}\|v\|_{X^{0,\frac{1}{2}+}}\|\mathcal{D}v\|_{X^{0,\frac{1}{2}+}}\big)
\big(\frac{1}{N_3^{\frac{1}{2}-}}\|\mathcal{D}v\|_{X^{0,\frac{1}{2}+}}\|v\|_{X^{0,\frac{1}{2}+}}\big)
\|v\|_{X^{\frac{1}{2}+,\frac{1}{2}+}}^2$$
\vspace{2mm}
$$\lesssim \frac{1}{(N_1^*)^{3-}}\|\mathcal{D}v\|_{X^{0,\frac{1}{2}+}}^2\|v\|_{X^{1,\frac{1}{2}+}}^4$$
\vspace{2mm}
\begin{equation}
\label{eq:BigCase1Case2Subcase1}
\lesssim \frac{1}{(N_1^*)^{3-}}E^2(u(0))
\end{equation}

\vspace{3mm}

\textbf{Subcase 2:} $\max\{|\xi_1|,|\xi_2|,|\xi_3|\} \gtrsim (N_1^*)^{\frac{1}{2}}.$

\vspace{2mm}

We consider WLOG when $|\xi_1| \gtrsim (N_1^*)^{\frac{1}{2}}$. The other two cases are analogous. Hence, we have to estimate:

$$\int_{\tau_1+\cdots+\tau_6=0} \int_{\xi_1+\cdots+\xi_6=0, |\xi_1+\xi_2+\xi_3|\sim N_1} \frac{1}{(N_1^*)^2} \theta(N_2) \theta(N_3)$$
\vspace{2mm}
$$|\widetilde{v}_{\gtrsim (N_1^*)^{\frac{1}{2}}}(\xi_1,\tau_1)||\widetilde{\bar{v}}(\xi_2,\tau_2)||\widetilde{v}(\xi_3,\tau_3)|
|\widetilde{\bar{v}}_{N_2}(\xi_4,\tau_4)|
|\widetilde{v}_{N_3}(\xi_5,\tau_5)||(\chi\bar{v})\,\widetilde{}_{N_4}\,(\xi_6,\tau_6)| d\xi_j d\tau_j$$
\vspace{2mm}
$$\lesssim \int_{\tau_1+\cdots+\tau_6=0} \int_{\xi_1+\cdots+\xi_6=0, |\xi_1+\xi_2+\xi_3|\sim N_1} \frac{1}{(N_1^*)^2}
|\widetilde{v}_{\gtrsim (N_1^*)^{\frac{1}{2}}}(\xi_1,\tau_1)||(\mathcal{D}\bar{v})\,\widetilde{}\,_{N_2}(\xi_4,\tau_4)|$$
\vspace{2mm}
$$\big(|(\mathcal{D}v)\,\widetilde{}\,_{N_3}(\xi_5,\tau_5)||(\chi\bar{v})\,\widetilde{}_{N_4}\,(\xi_6,\tau_6)|\big)
|\widetilde{\bar{v}}(\xi_2,\tau_2)||\widetilde{v}(\xi_3,\tau_3)| d\xi_j d\tau_j$$
\vspace{3mm}
We use an $L^4_{t,x},L^4_{t,x},L^2_{t,x},L^{\infty}_{t,x},L^{\infty}_{t,x}$ H\"{o}lder inequality, and Proposition \ref{chiImprovedStrichartz} to deduce that this expression is:

$$\lesssim \frac{1}{(N_1^*)^2} \|v_{\gtrsim (N_1^*)^{\frac{1}{2}}}\|_{X^{0,\frac{3}{8}+}} \|\mathcal{D}v\|_{X^{0,\frac{3}{8}+}} \big(\frac{1}{(N_3)^{\frac{1}{2}-}} \|\mathcal{D}v\|_{X^{0,\frac{1}{2}+}}
\|v\|_{X^{0,\frac{1}{2}+}}\big) \|v\|_{X^{\frac{1}{2}+,\frac{1}{2}+}}^2$$
\vspace{2mm}
$$\lesssim \frac{1}{(N_1^*)^{3-}} \|\mathcal{D}v\|_{X^{0,\frac{1}{2}+}}^2 \|v\|_{X^{1,\frac{1}{2}+}}^4$$
\vspace{2mm}

Here we used the fact that $\|v_{\gtrsim (N_1^*)^{\frac{1}{2}}}\|_{X^{0,\frac{3}{8}+}}\lesssim \frac{1}{(N_1^*)^{\frac{1}{2}}}\|v\|_{X^{1,\frac{1}{2}+}}$.

\vspace{2mm}

Hence, the contribution from this Subcase is:

\begin{equation}
\label{eq:BigCase1Case2Subcase2}
\lesssim \frac{1}{(N_1^*)^{3-}}E^2(u(0))
\end{equation}
\vspace{3mm}
Combining $(\ref{eq:BigCase1Case1Subcase1}),(\ref{eq:BigCase1Case1Subcase2}),(\ref{eq:BigCase1Case2Subcase1}),
(\ref{eq:BigCase1Case2Subcase2})$, it follows that the contribution to $I_{N_1,N_2,N_3,N_4}$ coming from Big Case 1 is:

\begin{equation}
\label{eq:BigCase1Bound}
O(\frac{1}{(N_1^*)^{3-}}E^2(u(0))).
\end{equation}

\vspace{4mm}

\textbf{Big Case 2:} We recall that in this Big Case $N_2^* \sim N_3^*$.

\vspace{3mm}

From Proposition \ref{Proposition 4.2}, we observe that in Big Case 2, one has:

\begin{equation}
\label{eq:M4BigCase2}
M_4(\xi_1,\xi_2,\xi_3,\xi_4)=O(\frac{1}{(N_1^*)^3}\theta(N_1^*)\theta(N_2^*)N_3^*N_4^*)
\end{equation}

\vspace{2mm}

In Big Case 2, we argue in the same way as we did for the Hartree equation on $S^1$ in \cite{SoSt1}. The same argument as in Section 4.1. of the mentioned paper implies that the contribution to $I_{N_1,N_2,N_3,N_4}$ coming from Big Case 2 is:

\begin{equation}
\label{eq:BigCase2Bound}
O(\frac{1}{(N_1^*)^3}E^2(u(0))).
\end{equation}

\vspace{2mm}

We refer the reader to the proof in \cite{SoSt1}. Let us note that in the periodic setting, we could only get a decay factor of $\frac{1}{(N_1^*)^2}$.

\vspace{2mm}

We use $(\ref{eq:BigCase1Bound})$,$(\ref{eq:BigCase2Bound})$,$(\ref{eq:NjstarE2})$, and sum in the $N_j^*$ to deduce Lemma \ref{Bigstar} for $t_0=0$. By time-translation, the general claim follows.
\end{proof}
\vspace{3mm}

\begin{remark}
\label{Remark 4.1}
If we use the method of proof of Theorem \ref{Theorem 1} for the Hartree equation (here we just use mass and energy as conserved quantities), we can obtain the bound $\|u(t)\|_{H^s} \leq C(1+|t|)^{(s-1)+}$, which is a weaker result than Theorem \ref{Theorem 2} when $s$ is large.
\end{remark}

\section{Appendix A: Auxiliary results for the cubic nonlinear Schr\"{o}dinger equation.}

In Appendix A, we prove Proposition \ref{Proposition 3.2}

\begin{proof}

By continuity of Energy and Mass on $H^1$, both of the claims clearly hold for $n=1$. For higher $n$, we will need to work directly with the higher conserved quantities of (\ref{eq:cubicnls}). One can explicitly compute these quantities by means of a recursive formula. The formula that we use comes from \cite{FadTak,ZM}. Let $u$ be a solution of (\ref{eq:cubicnls}). Let us define a sequence of polynomials $(P_k)_{k \geq 1}$ by:

\begin{equation}
\label{eq:recursionformula}
\begin{cases}
P_1:=|u|^2,\\
P_{k+1}:=-i \bar{u} \frac{\partial}{\partial x}(\frac{P_k}{\bar{u}}) + \sum_{l=1}^{k-1} P_l P_{k-l},\,\, \mbox{for}\, k\geq 1.
\end{cases}
\end{equation}
Then, for all $k\geq 1$, $\int P_k \, dx$ is a conserved quantity for $(\ref{eq:cubicnls})$.

\vspace{2mm}

For the details, we refer the reader to \cite{FadTak}, more precisely to Page 53, where it is noted that formulas (4.19),(4.20),(4.34) in the textbook still remain valid for our equation. Let us now explicitly compute:

$$P_2=-i\bar{u}\frac{\partial}{\partial x}u.$$

$$P_3=-\bar{u}\frac{\partial^2}{\partial x^2}u + \frac{1}{2}|u|^4.$$

$$P_4=i\bar{u}\frac{\partial^3}{\partial x^3}u - i|u|^2 \bar{u} \frac{\partial}{\partial x}u.$$

\vspace{2mm}

The conserved quantity corresponding to $P_1$ is:

$$\int P_1 \, dx = \int |u|^2 dx\,\,=\mbox{\emph{Mass}}.$$
For the conserved quantities corresponding to $P_2,P_3$, we integrate by parts to obtain:

$$\int P_2 \, dx = -\frac{i}{2} \int (\bar{u}\frac{\partial}{\partial x}u - u \frac{\partial}{\partial x}\bar{u})dx\,\,
\sim \mbox{\emph{Momentum}}.$$

$$\int P_3 \, dx = \int \big| \frac{\partial}{\partial x}u \big|^2 dx + \frac{1}{2} \int |u|^4 dx\,\, \sim \mbox{\emph{Energy}}.$$
So, we recover the well-known conserved quantities this way.

We argue by induction to deduce that:

$$P_n=c \bar{u}\frac{\partial^{n-1}}{\partial x^{n-1}}u + l.o.t.$$
Again, by induction, we obtain that each lower-order term contains in total at most $n-3$ derivatives. It follows that the conserved quantity we want to study is:

$$E_n(u):=\int P_{2n+1} \, dx=\pm c \int \big| \frac{\partial^n}{\partial x^n}u \big|^2 dx + l.o.t.$$
\vspace{2mm}
Here, each lower-order term is the integral of a polynomial in $x$-derivatives of $u,\bar{u}$ containing in total at most $2n-2$ derivatives. If we integrate by parts, we can arrange so that at most $n$ derivatives fall on one factor, and that at most $n-2$ derivatives fall on all the other factors combined. By using H\"{o}lder's inequality \footnote{we estimate the factor with the most  derivatives, and an arbitrary other factor in $L^2$; the rest of the factors we estimate in $L^{\infty}$} and by Sobolev embedding, there exists a polynomial $Q_n=Q_n(x)$ s.t.

\begin{equation}
\label{eq:Enbound1}
E_n(u) \geq C (\|u\|_{\dot{H}^n}^2- Q_n(\|u\|_{H^{n-1}})\|u\|_{H^n}).
\end{equation}

\vspace{3mm}
Similarly, if we also use multilinearity, it follows that there exists a polynomial $R_n$ in $(x,y)$ s.t.

$$|E_n(u)-E_n(v)| \lesssim (\|u\|_{H^n}+\|v\|_{H^n})\|u-v\|_{H^n}+$$

\begin{equation}
\label{eq:Enbound2}
R_n(\|u\|_{H^{n-1}},\|v\|_{H^{n-1}})\|u-v\|_{H^{n}}.
\end{equation}
\vspace{2mm}
The fact that $E_n$ is continuous on $H^n$ follows immediately from (\ref{eq:Enbound2}). This proves the first part of the claim.

\vspace{2mm}

Furthermore, if we define:

$$\widetilde{E_n}(u):=E_n(u)+\|u\|_{L^2}^2,$$
then, by (\ref{eq:Enbound1}), it follows that:

$$\widetilde{E_n}(u) \geq C_n (\|u\|_{H^n}^2- Q_n(\|u\|_{H^{n-1}})\|u\|_{H^n}).$$

\vspace{2mm}

This bound in turn implies:

\begin{equation}
\label{eq:Hnbound}
\|u\|_{H^n}\leq \frac{1}{2}\Big(Q_n(\|u\|_{H^{n-1}}) + \sqrt{(Q_n(\|u\|_{H^{n-1}}))^2+ \frac{4}{C_n} \widetilde{E_n}(u)}\,\Big).
\end{equation}
\vspace{2mm}
We finally define:

\begin{equation}
\label{eq:Bn}
B_n(\Phi):=\frac{1}{2}\Big(Q_n(B_{n-1}(\Phi)) + \sqrt{(Q_n(B_{n-1}(\Phi)))^2+\frac{4}{C_n}\widetilde{E_n}(\Phi)}\,\Big).
\end{equation}
We combine the fact that $E_n$ is continuous on $H^n$, conservation of mass, (\ref{eq:Hnbound}), and argue by induction to deduce the second part of the claim if we define $B_n$ as in (\ref{eq:Bn}).

\end{proof}

\section{Appendix B: Auxiliary results for the Hartree equation.}

We first prove Proposition \ref{Proposition 4.2} assuming Lemma \ref{Lemma 4.3}.

\begin{proof}
Let us first recall that:

\begin{equation}
\label{eq:VhatHartree}
\widehat{V} \in L^{\infty}
\end{equation}

\vspace{2mm}

As before, we consider $|\xi_j| \sim N_j$ for dyadic integers $N_1,N_2,N_3,N_4$. We order the $N_j$ to obtain $N_j^*$, for $j=1,\ldots,4$, s.t. $N_1^* \geq N_2^* \geq N_3^* \geq N_4^*$. Let's recall the localization $(\ref{eq:NjstarE2})$. By symmetry, let us also consider WLOG $N_1^* \sim N_1$.

\vspace{2mm}

We consider the following cases:

\vspace{3mm}

\textbf{Case 1:} $N_2^* \gg N_3^*.$

\vspace{2mm}

We must consider several subcases:

\vspace{2mm}

\textbf{Subcase 1:} $N_2^* \sim N_2.$

\vspace{2mm}

Since $\xi_1+\xi_2+\xi_3+\xi_4=0$, we obtain:

\begin{equation}
\label{eq:factorizationpropertyHartree}
|\xi_1^2-\xi_2^2+\xi_3^2-\xi_4^2|=2|(\xi_1+\xi_2)(\xi_1+\xi_4)|
\end{equation}

\vspace{2mm}

In this Subcase, this expression is:

$$\sim N_1^*|\xi_1+\xi_2|.$$
By Lemma \ref{Lemma 4.3}, we know:

$$|(\theta(\xi_1))^2-(\theta(\xi_2))^2| \lesssim ||\xi_1|-|\xi_2|| \frac{(\theta(\xi_1))^2}{|\xi_1|}
\lesssim |\xi_1+\xi_2| \frac{\theta(N_1^*)\theta(N_2^*)}{N_1^*}$$
Similarly, assuming WLOG that $|\xi_3| \geq |\xi_4|$, we use Lemma \ref{Lemma 4.3}, and the fact that $\frac{(\theta(\xi_3))^2}{|\xi_3|} \lesssim \frac{(\theta(\xi_1))^2}{|\xi_1|}$, if $|\xi_3| \geq N$, and $(\theta(\xi_3))^2-(\theta(\xi_4))^2=0$, if $|\xi_3| \leq N$ to deduce that:

$$|(\theta(\xi_3))^2-(\theta(\xi_4))^2| \lesssim |\xi_3+\xi_4| \frac{\theta(N_1^*)\theta(N_2^*)}{N_1^*}=
|\xi_1+\xi_2| \frac{\theta(N_1^*)\theta(N_2^*)}{N_1^*}$$
Combining the last three bounds and $(\ref{eq:VhatHartree})$, we obtain:

\begin{equation}
\label{eq:Case1Subcase1Hartree}
\Psi=O(\frac{1}{(N_1^*)^2} \theta(N_1^*)\theta(N_2^*)).
\end{equation}

\vspace{3mm}

\textbf{Subcase 2:} $N_2^* \sim N_3.$

\vspace{2mm}

In this subcase, we don't expect any cancelation in neither the numerator nor the denominator. So, we just estimate the numerator as $O(\theta(N_1^*)\theta(N_2^*))$, and we estimate the denominator as $\sim (N_1^*)^2$.
Consequently:

\begin{equation}
\label{eq:Case1Subcase2Hartree}
\Psi=O(\frac{1}{(N_1^*)^2} \theta(N_1^*)\theta(N_2^*)).
\end{equation}

\vspace{4mm}

\textbf{Case 2:} $N_2^* \sim N_3^*.$

\vspace{2mm}

As before, we consider two subcases:

\vspace{2mm}

\textbf{Subcase 1:} $N_3^* \gg N_4^*.$

\vspace{2mm}

It suffices to WLOG consider when $N_2^* \sim N_2, N_3^* \sim N_3, N_4^* \sim N_4$.

\vspace{2mm}

We have:

$$|\xi_1^2-\xi_2^2+\xi_3^2-\xi_4^2|=2|(\xi_1+\xi_2)(\xi_1+\xi_4)| \sim N_1^*|\xi_1+\xi_2|.$$
We argue now as in Subcase 1 of Case 1 to obtain:

$$\Psi=O\big(\frac{1}{(N_1^*)^2} \theta(N_1^*)\theta(N_2^*)\big).$$
Since $N_3^* \sim N_1^*$, in this subcase, we obtain:

\begin{equation}
\label{eq:Case2Subcase1Hartree}
\Psi=O\big(\frac{1}{(N_1^*)^3} \theta(N_1^*)\theta(N_2^*)N_3^*\big).
\end{equation}

\vspace{2mm}

\textbf{Subcase 2:} $N_1^* \sim N_2^* \sim N_3^* \sim N_4^*.$

\vspace{2mm}

We know:

$$|\xi_1^2-\xi_2^2+\xi_3^2-\xi_4^2| \sim |(\xi_1+\xi_2)(\xi_1+\xi_4)|.$$
We must consider several sub-subcases.

\vspace{2mm}

\textbf{Sub-subcase 1:} $|\xi_1+\xi_2| \ll 1,\,\,|\xi_1+\xi_4| \ll 1.$

\vspace{2mm}

Since $\xi_1+\xi_2+\xi_3+\xi_4=0$, we get:

$$\xi_3+(\xi_1+\xi_2)=-\xi_4$$

$$\xi_3+(\xi_1+\xi_4)=-\xi_2$$

$$\xi_3+(\xi_1+\xi_2)+(\xi_1+\xi_4)=\xi_1.$$
From the previous identities, the Double Mean Value Theorem $(\ref{eq:DoubleMVT})$, and $(\ref{eq:thetadoubleprime})$, we obtain that:

$$|(\theta(\xi_1))^2-(\theta(\xi_2))^2+(\theta(\xi_3))^2-(\theta(\xi_4))^2|=$$
\vspace{2mm}
$$=|(\theta(\xi_3+(\xi_1+\xi_2)+(\xi_1+\xi_4)))^2-(\theta(\xi_3+(\xi_1+\xi_4)))^2+(\theta(\xi_3+(\xi_1+\xi_2)))^2
-(\theta(\xi_4))^2|$$
\vspace{2mm}
$$\lesssim |\xi_1+\xi_2| |\xi_1+\xi_4| |(\theta^2)''(\xi_3)|\lesssim |\xi_1+\xi_2||\xi_1+\xi_4|\,\frac{(\theta(\xi_3))^2}{|\xi_3|^2}$$
\vspace{2mm}
$$\lesssim |\xi_1+\xi_2||\xi_1+\xi_4| \frac{\theta(N_1^*) \theta(N_2^*)}{(N_1^*)^2}$$
\vspace{2mm}
So, in this sub-subcase, we obtain that:

$$\Psi=O(\frac{1}{(N_1^*)^2} \theta(N_1^*)\theta(N_2^*))=$$
\vspace{2mm}
\begin{equation}
\label{eq:Case2Subcase2Subsubcase1Hartree}
=O(\frac{1}{(N_1^*)^4}\theta(N_1^*)\theta(N_2^*)N_3^*N_4^*)
\end{equation}

\vspace{2mm}

\textbf{Sub-subcase 2:} $|\xi_1+\xi_4| \gtrsim 1.$

\vspace{2mm}

Here:

$$|\xi_1^2-\xi_2^2+\xi_3^2-\xi_4^2|= 2|\xi_1+\xi_2||\xi_1+\xi_4| \gtrsim |\xi_1+\xi_2|=|\xi_3+\xi_4|.$$
Hence, by Lemma \ref{Lemma 4.3}:

$$\Psi=O\big( \frac{|(\theta(\xi_1))^2-(\theta(\xi_2))^2}{|\xi_1+\xi_2|} + \frac{|(\theta(\xi_3))^2-(\theta(\xi_4))^2|}{|\xi_3+\xi_4|}\big)=$$
\vspace{2mm}
$$=O(\frac{1}{|\xi_1|}(\theta(\xi_1))^2)=O(\frac{1}{N_1^*}\theta(N_1^*)\theta(N_2^*))=$$
\vspace{2mm}
\begin{equation}
\label{eq:Case2Subcase2Subsubcase2Hartree}
O(\frac{1}{(N_1^*)^3} \theta(N_1^*)\theta(N_2^*)N_3^*N_4^*)
\end{equation}

\vspace{2mm}

\textbf{Sub-subcase 3:} $|\xi_1+\xi_2| \gtrsim 1.$

\vspace{2mm}

We group the terms in the numerator as:

$$\big((\theta(\xi_1))^2-(\theta(\xi_4))^2\big)+\big((\theta(\xi_3))^2-(\theta(\xi_2))^2\big)$$
Then, we argue as in the previous sub-subcase to obtain:

\vspace{2mm}
\begin{equation}
\label{eq:Case2Subcase2Subsubcase3Hartree}
O\big(\frac{1}{(N_1^*)^3} \theta(N_1^*)\theta(N_2^*)N_3^*N_4^*\big)
\end{equation}

\end{proof}

Let us now prove Lemma \ref{Lemma 4.3}.

\begin{proof}

We have to consider five cases:

\begin{enumerate}
\item $N \leq |y|,\,2N \leq |x|$
\item $N \leq |y| \leq |x| \leq 2N$
\item $|y| \leq |x| \leq N$
\item $|y| \leq N,\, 2N \leq |x|$
\item $|y| \leq N \leq |x| \leq 2N$
\end{enumerate}

\vspace{2mm}

We consider each case separately:

\vspace{3mm}

\begin{enumerate}
\item $ |(\theta(x)^2-(\theta(y))^2|\, \leq (|x|-|y|)\, \sup_{[|y|,|x|]} |(\theta^2)'(z)| \leq (|x|-|y|)\, \sup_{[N,|x|]}|(\theta^2)'(z)|$

\vspace{2mm}

By using $(\ref{eq:thetaprime})$, this expression is:

$$\lesssim (|x|-|y|) \sup_{[N,|x|]}\big(\frac{(\theta(z))^2}{|z|}\big) \lesssim
(|x|-|y|) \sup_{[2N,|x|]}\big(\frac{(\theta(z))^2}{|z|}\big)=$$

$$=(|x|-|y|) \sup_{[2N,|x|]}\frac{|z|^{2s-1}}{N^{2s}}= (|x|-|y|)\frac{|x|^{2s-1}}{N^{2s}}
= (|x|-|y|)\frac{(\theta(x))^2}{|x|}.$$

\vspace{2mm}

\item $|(\theta(x))^2-(\theta(y))^2| \leq (|x|-|y|) \sup_{[|y|,|x|]}|(\theta^2)'(z)|
\lesssim (|x|-|y|) \sup_{[|y|,|x|]}\big(\frac{(\theta(z))^2}{|z|}\big)$

For $z \in [|y|,|x|]$, one has:

$$\frac{(\theta(z))^2}{|z|}\sim \frac{(\theta(N))^2}{|N|} \sim \frac{(\theta(x))^2}{|x|}.$$

Hence, we get the wanted bound in this case.

\vspace{2mm}

\item In this case: $(\theta(x))^2-(\theta(y))^2=0.$

\vspace{2mm}

\item $|(\theta(x))^2-(\theta(y))^2| = |\frac{|x|^{2s}}{N^{2s}}-(\theta(N))^2|$, and we argue as in the first case.

\vspace{2mm}

\item $|(\theta(x))^2-(\theta(y))^2| = |(\theta(x))^2-(\theta(N))^2|$, and we argue as in the second case.

\end{enumerate}
\vspace{2mm}

Lemma \ref{Lemma 4.3} now follows.

\end{proof}

\vspace{3mm}

\section{Appendix C: The derivative nonlinear Schr\"{o}dinger equation:}

\vspace{2mm}

In this Appendix, we give a brief sketch of the proof of $(\ref{eq:dnlsbound})$.

\vspace{2mm}

We don't consider derivative nonlinear Schr\"{o}dinger equation directly. Rather, we argue as in \cite{Hay,HayOz1,HayOz2}, and we apply to $(\ref{eq:dnls})$ the following gauge transform:

\begin{equation}
\label{eq:gaugetransform}
\mathcal{G}f(x):=e^{-i\int_{-\infty}^x |f(y)|^2dy}f(x).
\end{equation}

\vspace{2mm}

For $u$ a solution of (\ref{eq:dnls}),we take $w:=\mathcal{G}u.$
Then, it can be shown that $w$ solves:

\begin{equation}
\label{eq:gaugednls}
\begin{cases}
i w_t + \Delta w=-i w^2\bar{w}_x -\frac{1}{2}|w|^4w \\
w(x,0)=w_0(x)=\mathcal{G}\Phi(x),x \in \mathbb{R},t \in \mathbb{R}.
\end{cases}
\end{equation}
\vspace{2mm}
The equation (\ref{eq:gaugednls}) has as a corresponding Hamiltonian:

\begin{equation}
\label{eq:Hamiltonian}
E(f):=\int \partial_x f \partial_x \bar{f}dx - \frac{1}{2} Im \int f \bar{f} f \partial_x \bar{f}dx.
\end{equation}
\vspace{2mm}
Although the problem is not defocusing a priori, in \cite{CKSTT},\cite{CKSTT5}, it is noted that the smallness condition (\ref{eq:smallnessassumption}) guarantees that the energy $E(w(t))$ is positive and that it gives us a priori bounds on $\|w(t)\|_{H^1}$.

\vspace{3mm}

It can be shown that the gauge transform satisfies the following boundedness property:

\newtheorem*{Gaugetransformbound}{Gauge transform bound}

\begin{Gaugetransformbound}

For $s \geq 1$, there exists a polynomial $P_s=P_s(x)$ such that:

$$\|\mathcal{G}f\|_{H^s} \leq P_s(\|f\|_{H^1})\|f\|_{H^s},$$

$$\|\mathcal{G}^{-1}f\|_{H^s}\leq P_s(\|f\|_{H^1})\|f\|_{H^s}.$$

\end{Gaugetransformbound}

\vspace{2mm}

From the bi-continuity of gauge transform, and the uniform bounds on $\|w(t)\|_{H^1}$, it suffices to prove for solutions of (\ref{eq:gaugednls}) the bounds that we want to hold for solutions of the derivative NLS.

\vspace{2mm}

One can show that a local-in-time estimate, analogous to Proposition \ref{Proposition 4.1}, holds for $(\ref{eq:gaugednls})$. The key is to use the following:

\newtheorem*{trilinearestimate}{Trilinear Estimate}

\begin{trilinearestimate}

Let $s \geq 1, b \in (\frac{1}{2},\frac{5}{8}],b'>\frac{1}{2}$, then for $v_1,v_2,v_3:
\mathbb{R} \times \mathbb{R} \rightarrow \mathbb{C}$, the following estimate holds:

$$\|v_1 v_2 \overline{(v_3)}_x \|_{X^{s,b-1}} \lesssim
\|v_1\|_{X^{1,b'}}\|v_2\|_{X^{1,b'}}\|v_3\|_{X^{s,b'}}$$
\begin{equation}
\label{eq:trilinearniestimate}
+\|v_1\|_{X^{1,b'}}\|v_2\|_{X^{s,b'}}\|v_3\|_{X^{1,b'}}
+\|v_1\|_{X^{s,b'}}\|v_2\|_{X^{1,b'}}\|v_3\|_{X^{1,b'}}.
\end{equation}

\end{trilinearestimate}

\vspace{2mm}

This estimate is the analogue of Proposition 2.4. in \cite{Tak}, where the identical statement is proved in the context of low regularities. The proof for $s\geq 1$ is similar, with minor modifications.

\vspace{2mm}

We now argue as in Theorem \ref{Theorem 2}, by using the technique of \emph{higher modified energies}. We define $E^1$ as before. We consider the higher modified energy $E^2$ given by:

\begin{equation}
\label{eq:E2dnls}
E^2(w):=E^1(w)+\lambda_4(M_4;w).
\end{equation}

Using the equation $(\ref{eq:gaugednls})$, it follows that a good choice for the multiplier $M_4$ on the set $\Gamma_4$ is:

\begin{equation}
\label{eq:M4dnls}
M_4 \sim \frac{(\theta(\xi_1))^2\xi_3+(\theta(\xi_2))^2\xi_4+(\theta(\xi_3))^2\xi_1+(\theta(\xi_4))^2\xi_2}
{\xi_1^2-\xi_2^2+\xi_3^2-\xi_4^2}\,\,
\end{equation}

\vspace{2mm} We define the ordered dyadic localizations $N_j^*$ as before. With this notation, one can show the following:

\newtheorem*{M4bound}{Multiplier bound}

\begin{M4bound}
On $\Gamma_4$, one has the pointwise bound:
\begin{equation}
\label{eq:M4bounddnls}
|M_4|\lesssim \frac{1}{N_1^*}\theta(N_1^*)\theta(N_2^*).
\end{equation}
\end{M4bound}

By construction of $M_4$, we obtain:

$$\frac{d}{dt}E^2(w(t))=\frac{d}{dt}E^1(w)+\frac{d}{dt}
\lambda_4(M_4;w)$$

\begin{equation}
\label{eq:E2primednls}
=\lambda_6(\sigma_6;w)+\lambda_6(M_6;w)+\lambda_8(M_8;w).
\end{equation}

Here:

\begin{equation}
\label{eq:sigma6dnls}
\sigma_6 = (\theta(\xi_1))^2-(\theta(\xi_2))^2+(\theta(\xi_3))^2-(\theta(\xi_4))^2+(\theta(\xi_5))^2-(\theta(\xi_6))^2.
\end{equation}

$$M_6 \sim M_4(\xi_{123},\xi_4,\xi_5,\xi_6)\xi_2+M_4(\xi_1,\xi_{234},\xi_5,\xi_6)\xi_3$$
\begin{equation}
\label{eq:M6dnls}
+M_4(\xi_1,\xi_2,\xi_{345},\xi_6)\xi_4 + M_4(\xi_1,\xi_2,\xi_3,\xi_{456})\xi_5.
\end{equation}

$$M_8 \sim M_4(\xi_{12345},\xi_6,\xi_7,\xi_8)-M_4(\xi_1,\xi_{23456},\xi_7,\xi_8)$$
\begin{equation}
\label{eq:M8dnls}
+M_4(\xi_1,\xi_2,\xi_{34567},\xi_8)-M_4(\xi_1,\xi_2,\xi_3,\xi_{45678}).
\end{equation}

\vspace{2mm}

Using $(\ref{eq:E2primednls})$ and $(\ref{eq:M4bounddnls})$, we can argue similarly as in the proof of Theorem \ref{Theorem 2} to deduce that:

\begin{equation}
\label{eq:E2incrementdnls}
|E^2(w(t_0+\delta))-E^2(w(t_0))|\lesssim \frac{1}{N^{\frac{1}{2}-}} E^2(w(t_0)).
\end{equation}

\vspace{2mm}

The bound for the derivative NLS follows from $(\ref{eq:E2incrementdnls})$. $\Box$


\begin{thebibliography}
{References:}
\bibitem{AGT} R. Adami, F. Golse, A. Teta, \emph{Rigorous Derivation of the Cubic NLS in Dimension One}, Journal of Statistical Physics, \textbf{127}, no. 6 (2007), 1193-1220.
\bibitem{BN} D. Benney, A. Newell, \emph{Random wave closures}, Stud. Appl. Math. \textbf{48} (1969), 29-53.
\bibitem{BS} D. Benney, P. Saffman, \emph{Nonlinear interactions of random waves in a dispersive medium}, Proc. Roy. Soc. A \textbf{289} (1966), 301-320.
\bibitem{B} J. Bourgain, \emph{Fourier transform restriction phenomena for certain lattice subsets and applications
to nonlinear evolution equations, I:Schr\"{o}dinger equations}, Geom. Funct. Anal. \textbf{3} (1993), 107-156.
\bibitem{B2} J. Bourgain, \emph{On the growth in time of higher
Sobolev norms of smooth solutions of Hamiltonian
PDE}, Int. Math. Resarch Notices ,\textbf{6} (1996), 277-304.
\bibitem{B7} J. Bourgain, \emph{Refinements of Strichartz's Inequality and applications to 2D-NLS with critical nonlinearity}, Int. Math. Research Notices, \textbf{5} (1998), 253-283.
\bibitem{B3} J. Bourgain, \emph{Global Solutions of Nonlinear Schr\"{o}dinger Equations}, AMS Colloquium
Publications, \textbf{46}, AMS,Providence,RI, 1999.
\bibitem{B5} J. Bourgain, \emph{On growth of Sobolev norms in linear Schr\"{o}dinger equations with smooth time dependent
potential}, J. Anal. Math., \textbf{77}(1999),315-348.
\bibitem{B6} J. Bourgain, \emph{Growth of Sobolev norms in linear Schr\"{o}dinger equations with quasi-periodic potential}, Comm. Math. Phys. \textbf{204} (1999),no.1., 207-247.
\bibitem{B4} J. Bourgain, \emph{Remarks on stability and diffusion in high-dimensional Hamiltonian systems and
partial differential equations}, Ergod. Th. and Dynam. Sys. \textbf{24} (2004), 1331-1357.
\bibitem{BGT} N. Burq, P. G\'{e}rard, N. Tzvetkov, \emph{Bilinear eigenfunction estimates and the nonlinear Schr\"{o}dinger equation on surfaces}, Invent. Math. \textbf{159} (2005), 187-223.
\bibitem{CafE} L. Caffarelli, W. E, Editors, \emph{Hyperbolic Equations and
Frequency Interactions, Lecture notes by Jean Bourgain on Nonlinear Schr\"{o}dinger Equations}, IAS/Park City Mathematics Series, Vol. \textbf{5}, AMS (1998).
\bibitem{CatW} F. Catoire, W.-M. Wang \emph{Bounds on Sobolev norms for the nonlinear Schr\"{o}dinger equation on general tori}, preprint (2008), \texttt{http://arxiv.org/abs/0809.4633}
\bibitem{Ca} T. Cazenave \emph{Semilinear Schr\"{o}dinger Equations}, Courant Lecture Notes, \textbf{10}, AMS, CIMS, 2003.
\bibitem{CKSTT} J. Colliander, M. Keel, G. Staffilani, H. Takaoka, T. Tao, \emph{Global well-posedness for Schr\"{o}dinger equations with derivative}, SIAM J. Math. Anal. \textbf{33} (2001), no.3, 649-669.
\bibitem{CKSTT2} J. Colliander, M. Keel, G. Staffilani, H. Takaoka, T. Tao, \emph{Polynomial upper bounds for the orbital instability of the 1D cubic NLS below the energy norm}, Commun. Pure Appl. Anal. \textbf{2} (2003), no. 1, 33-50.
\bibitem{CKSTT3} J. Colliander, M. Keel, G. Staffilani, H. Takaoka, T. Tao, \emph{Sharp global well-posedness for KdV and modified KdV on $\mathbb{R}$ and $\mathbb{T}$}, J. Amer. Math. Soc. \textbf{16} no.3 (2003), 705-749. (electronic)
\bibitem{CKSTT4} J. Colliander, M. Keel, G. Staffilani, H. Takaoka, T. Tao, \emph{Multilinear estimates for periodic KdV equations, and applications}, J. Funct. Anal. \textbf{211} no. 1 (2004), 173-218.
\bibitem{CKSTT5} J. Colliander, M. Keel, G. Staffilani, H. Takaoka, T. Tao, \emph{A refined global well-posedness for Schr\"{o}dinger equations with derivative}, SIAM J. Math. Analysis, \textbf{34} no. 1 (2002), 64-86.
\bibitem{CKSTT6} J. Colliander, M. Keel, G. Staffilani, H. Takaoka, T. Tao, \emph{Weakly turbulent solutions for the cubic defocusing Nonlinear Schr\"{o}dinger Equation}, preprint, 2008 \texttt{http://arxiv.org/abs/0808.1742}.
\bibitem{DW} \emph{Dispersive Wiki Page, Cubic nonlinear Schr\"{o}dinger equation}, \texttt{http://tosio.math.toronto.edu/wiki}
\bibitem{De} J.-M. Delort, \emph{Growth of Sobolev norms of solutions of linear Schr\"{o}dinger equations on some compact manifolds}, preprint, International Mathematics Research Notices Advance Access (2009).
\bibitem{D} J. Duoandikoetxea, \emph{Fourier Analysis}, Graduate Studies in Mathematics, \textbf{29}, AMS, Providence, RI 2000
\bibitem{FadTak} L. Faddeev, L. Takhtajan, \emph{Hamiltonian Methods in the Theory of Solitons}, Springer Classics in Mathematics, Springer-Verlag, (1987)
\bibitem{FL} F. Fr\"{o}hlich, E. Lenzmann, \emph{Mean-Field Limit of Quantum Bose Gases and Nonlinear Hartree Equation}, S\'{e}minaire \'{E}.D.P. (2003-2004), Expos\'{e} XVIII, 26 p.
\bibitem{GiOz} J. Ginibre, T. Ozawa, \emph{Long-range scattering for Non-Linear Schr\"{o}dinger and Hartree Equations in Space Dimension $n\geq 2$}, Comm. Math. Phys. \textbf{151} (1993), 619-645.
\bibitem{GiVe} J. Ginibre, G. Velo, \emph{Long Range Scattering and Modified Wave Operators for some Hartree Type Equations III, Gevrey spaces and low dimensions}, J. Diff. Eq. \textbf{175} issue 2 (2001), 415-501.
\bibitem{G} A. Gr\"{u}nrock, \emph{On the Cauchy- and periodic boundary value problem for a certain class of derivative nonlinear Schr\"{o}dinger equations}, preprint \texttt{arXiv:math/0006195}
\bibitem{Hay} N. Hayashi, \emph{The initial value problem for the derivative nonlinear Schr\"{o}dinger equation in the energy space}, Nonlinear Anal. \textbf{20} (1993), 823-833.
\bibitem{HNO} N. Hayashi, P. Naumkin, T. Ozawa, \emph{Scattering Theory for the Hartree Equation}, Hokkaido University Preprints, Series \textbf{358}, Nov. (1996)
\bibitem{HayOz1} N. Hayashi, T. Ozawa, \emph{Finite energy solutions of nonlinear Schr\"{o}dinger equations of derivative type}, SIAM J. Math. Anal. \textbf{25} (1994), 1488-1503.
\bibitem{HayOz2} N. Hayashi, T. Ozawa, \emph{Remarks on nonlinear Schr\"{o}dinger equations in one space dimension}, Diff. Integral Eqs. \textbf{2} (1994), 1488-1503.
\bibitem{KN} D.J. Kaup, A.C. Newell, \emph{An exact solution for a Derivative Nonlinear Schr\"{o}dinger Equation}, J. Math. Phys., \textbf{19} no. 4 (1978), 798-801.
\bibitem{KPV4} C. Kenig, G. Ponce, L. Vega, \emph{Oscillatory integrals and regularity of dispersive equations}, Indiana Univ. Math. J. \textbf{40} (1991), 33-69.
\bibitem{KPV2} C. Kenig, G. Ponce, L. Vega, \emph{The Cauchy problem for the Korteweg-de Vries equation in
Sobolev spaces of negative indices}, Duke Math. J. \textbf{71},
(1993), 1-21.
\bibitem{KPV} C. Kenig, G. Ponce, L. Vega, \emph{A bilinear estimate with applications to the KdV Equation,}
J. Amer. Math. Soc. \textbf{9} (1996), 573-603.
\bibitem{KPV3} C. Kenig, G. Ponce, L. Vega, \emph{Quadratic Forms for the 1-D Semilinear Schr\"{o}dinger
Equation}, Transactions of the AMS, \textbf{348}, no.8 (1996), 3323-3353.
\bibitem{KSchSta} K. Kirkpatrick, B. Schlein, G. Staffilani, \emph{Derivation of the two dimensional nonlinear Schr\"{o}dinger equation from many body quantum dynamics}, preprint (2009), http://arxiv.org/abs/0808.0505.
\bibitem{LSY} E.H. Lieb, R. Seiringer, J. Yngvason, \emph{One-dimensional behaviour of dilute, trapped Bose gases}, Comm. Math. Phys. \textbf{244}, no.2 (2004), 347-393.
\bibitem{ZM} S. Manakov, V.E. Zakharov, \emph{The complete integrability of the nonlinear Schr\"{o}dinger equation}, Teoret. Mat. Fiz \textbf{19} (1974), 332-343.
\bibitem{MXZ1} C. Miao, G. Xu, L. Zhao, \emph{Global well-posedness and scattering for the energy critical, defocusing Hartree equation for radial data}, J. Funct. Anal. \textbf{253}, issue 2, (2007), 605-627.
\bibitem{MXZ2} C. Miao, G. Xu, L. Zhao, \emph{Global well-posedness and scattering for the energy-critical, defocusing Hartree equation in $\mathbb{R}^{1+n}$}, J. of Funct. Anal., \textbf{253}, 2, (2007), 605-627.
\bibitem{Sch} B. Schlein, \emph{Derivation of Effective Evolution Equations from Microscopic Quantum Dynamics}, Lecture Notes, Clay Summer School on Evolution Equations, Zurich, (2008).
\bibitem{SoSt1} V. Sohinger, \emph{Bounds on the growth of high Sobolev norms of solutions to Nonlinear Schr\"{o}dinger Equations on $S^1$}, preprint, (2010).
\bibitem{SoSt2} V. Sohinger, \emph{Bounds on the growth of high Sobolev norms of solutions to 2D Hartree Equations}, preprint, (2010).
\bibitem{S} G. Staffilani, \emph{On the growth of high
Sobolev norms of solutions for KdV and Schr\"{o}dinger equations}, Duke
Math. J. \textbf{86} no. 1 (1997), 109-142.
\bibitem{S2} G. Staffilani, \emph{Quadratic Forms for a 2-D Semilinear
Schr\"{o}dinger Equation}, Duke Math. J., \textbf{86},no.1 (1997), 79-107.
\bibitem{SuSu} C. Sulem, P.-L. Sulem, \emph{The nonlinear Schr\"{o}dinger equation}, Applied Math. Sciences, \textbf{139},
Springer-Verlag (1999).
\bibitem{Tak} H. Takaoka, \emph{Global well-posedness for Schr\"{o}dinger equations with derivative in a nonlinear term and data in low-order Sobolev spaces}, Electron. J. Diff. Eqns., \textbf{42} (2001), 1-23.
\bibitem{Tao} T. Tao, \emph{Nonlinear Dispersive Equations: Local and
global analysis}, CBMS Reg. Conf. Series in Math., \textbf{106},AMS,
Providence, RI, 2006.
\bibitem{W} W.-M. Wang, \emph{Logarithmic bounds on Sobolev norms for the time dependent linear Schr\"{o}dinger equation},
 preprint, 2008.\texttt{http://arxiv.org/abs/0805.3771}.
\bibitem{Zak} V.E. Zakharov, \emph{Stability of periodic waves of finite amplitude on a surface of deep fluid}, J. Appl. Mech. Tech. Phys., \textbf{9} (1968), 190-194.
\bibitem{Z} S.-J. Zhong, \emph{The growth in time of higher Sobolev norms of solutions to Schr\"{o}dinger
equations on compact Riemannian manifolds}, J.Differential Equations, \textbf{245} (2008), 359-376.


\end{thebibliography}
\end{document}